\documentclass[11pt, oneside]{amsart}   	% use "amsart" instead of "article" for AMSLaTeX format
\usepackage{geometry}                		% See geometry.pdf to learn the layout options. There are lots.
\usepackage[utf8]{inputenc}
\usepackage{graphicx}				% Use pdf, png, jpg, or eps§ with pdflatex; use eps in DVI mode
\usepackage{color}					% TeX will automatically convert eps --> pdf in pdflatex		
\usepackage{amssymb, amsmath}
\usepackage{bbold}
\usepackage{mathtools}
\usepackage{placeins}
\usepackage{hyperref}
\usepackage{float}
\usepackage{tikz-cd}
\usepackage[capitalize]{cleveref}
\newtheorem{theorem}{Theorem}[section]
\newtheorem{corollary}[theorem]{Corollary} 
\newtheorem{lemma}[theorem]{Lemma}
\newtheorem{proposition}[theorem]{Proposition}

\theoremstyle{definition}
\newtheorem{definition}[theorem]{Definition}
\newtheorem{remark}[theorem]{Remark}

\newtheorem{problem}[theorem]{Problem}

\floatstyle{boxed}\newfloat{Algorithm}{ht}{alg}
\crefname{Algorithm}{Algorithm}{Algorithms}

\newcommand{\Rep}{{\rm Rep}}
\newcommand{\mat}{\operatorname{Mat}}
\newcommand{\RR}{\mathbb{R}}
\newcommand{\R}{\mathbb{R}}
\newcommand{\QQ}{\mathbb{Q}}

\newcommand{\C}{\mathbb{C}}

\newcommand{\Z}{\mathbb{Z}}
\newcommand{\N}{\mathbb{N}}
\newcommand{\SL}{\operatorname{SL}}
\newcommand{\tr}{\operatorname{Tr}}

\newcommand{\poly}{\operatorname{poly}}

\newcommand{\ot}{\otimes}

\newcommand{\htheta}{\widehat{\Theta}}
\newcommand{\PD}{{\rm PD}}

\newcommand{\End}{{\rm End}}
\newcommand{\ext}{{\rm ext}}
\newcommand{\Ext}{{\rm Ext}}
\newcommand{\opp}{{\rm opp}}

\newcommand{\Mat}{\operatorname{Mat}}
\newcommand{\GL}{\operatorname{GL}}
\newcommand{\SI}{\operatorname{SI}}

\title{iPCA and stability of star quivers}
\author{Cole Franks and Visu Makam}
\begin{document}

\begin{abstract}
Integrated principal components analysis, or iPCA, is an unsupervised learning technique for grouped vector data recently defined by Tang and Allen \cite{tang2018integrated}. Like PCA, iPCA computes new axes that best explain the variance of the data, but iPCA is designed to handle corrupting influences by the elements within each group on one another - e.g. data about students at a school grouped into classrooms. Tang and Allen showed empirically that regularized iPCA finds useful features for such grouped data in practice. However, it is not yet known when unregularized iPCA generically exists. For contrast, PCA (which is a special case of iPCA) typically exists whenever the number of data points exceeds the dimension. We study this question and find that the answer is significantly more complicated than it is for PCA. Despite this complexity, we find simple sufficient conditions for a very useful case - when the groups are no more than half as large as the dimension and the total number of data points exceeds the dimension, iPCA generically exists. We also fully characterize the existence of iPCA in case all the groups are the same size. When all groups are not the same size, however, we find that the group sizes for which iPCA generically exists are the integral points in a non-convex union of polyhedral cones. Nonetheless, we exhibit an algorithm to decide whether iPCA generically exists that is polynomial in the node dimensions (based on the affirmative answer for the saturation conjecture by \cite{knutson1999honeycomb}) as well as a very simple randomized algorithm. 

At its core, our approach identifies connections between iPCA and stability notions for star quivers, thus bringing tools from invariant theory and quiver representations to the table. Approaching invariant theory and quiver representations from a computational perspective is a very rich endeavor and has found many important applications in complexity theory and algebraic statistics \cite{IQS2,GGOW16,DM,forbes2013explicit,AKRS,derksen2020maximum} in the last decade. Our work identifies another interesting and important algorithmic problem in the invariant theory of quivers, i.e., given a dimension vector $\alpha$ and a weight $\sigma$ for a quiver $Q$, decide if $\alpha$ is $\sigma$-semistable/polystable/stable. While current techniques are insufficient to give a polynomial time algorithm in any reasonable generality, we are able to leverage several interesting features of the representations of star quivers (the setting relevant to iPCA) to obtain algorithms that are polynomial time in the node dimensions by appealing to powerful results in optimization. In the future, we hope to build on these techniques to give polynomial time algorithms for deciding $\sigma$ stabilities of a dimension vector in greater generality. 
\end{abstract}
\maketitle

\setcounter{tocdepth}{1}
\tableofcontents

%\subsection{}
\section{Introduction}

Connections between stability notions in invariant theory and maximum likelihood estimation for certain Gaussian models was discovered recently in \cite{AKRS}. As one often finds with such an unexpected and deep connection, this allows for the infusion of techniques from invariant theory in statistics. This has already led to new and exciting results in statistics in a short period of time, see \cite{derksen2020maximum, derksen2020maximum2,franks2021near,franks2020rigorous}, with more results sure to come in the future as it remains an active area of research. 

Following in this spirit, we exploit the connection between invariant theory and maximum likelihood estimation to characterize and efficiently decide the existence of iPCA, a recently published unsupervised learning method. 
%, and in the process obtain new results about quiver representations. \CF{I 
Interestingly, the nature of the problems we pursue with respect to iPCA identifies a new algorithmic direction in the invariant theory of quivers giving evidence that such a connection can be beneficial in both directions. In particular, we obtain new results in quiver representations. Thus, our results are of interest to both statisticians and quiver theorists alike. We present our main results in a fashion that can be appreciated by a wide audience, yet we do not shy away from rigorously presenting the techniques we develop in the language of quiver representations - we would like this paper to serve as a motivation for quiver theorists to pursue the algorithmic problems we identify and build on our results. We include precise references to the literature for quiver theory that enables any reader to understand the relevant results from quivers should they choose to do so. To this end, our introduction begins with a very accessible discussion of iPCA and our results, followed by a short indication of our approach via quivers, the details of which will form the core of the rest of the paper.  

\subsection{Integrated Principal Components Analysis}
Integrated principal components analysis, or iPCA, is an unsupervised learning technique for grouped vector data recently defined by Tang and Allen \cite{tang2018integrated}. Like PCA, iPCA computes new axes that best explain the variance of the data, but iPCA is designed to handle corrupting influences by the elements within each group on one another. In particular, the axes chosen by iPCA remain the same even when each data point is replaced by an unknown linear combination of the data points within its group. Suppose the $n$ data points in $x_1, \dots, x_n \in \RR^p$ are split into $k$ groups, with the $i^{th}$ group of size $q_i$. Let $B_i$ be the $p \times q_i$ matrix whose columns are the data points in the $i^{th}$ group. See \cref{fig:groups}
\FloatBarrier

The iPC scores $u_i \in \RR^{d}$ are defined as the eigenvectors (in increasing order of eigenvalue) of the positive-definite matrix $\htheta$ defined as follows. For reasons to be explained shortly, we call this the \emph{ipca likelihood}.
%$$ \Theta  = \arg\min_{\Theta \succ 0}\frac{\prod_{i = 1}^n \det (B_i^\top \Theta B_i)}{\det(\Theta)^{n/d}}.$$
\begin{align}\htheta, \htheta_1, \dots, \htheta_k  &= \arg\max f(\Theta, \Theta_1, \dots, \Theta_m)\\
\text{ where } f(\Theta, \Theta_1, \dots, \Theta_k) &= n \log|\Theta| + p \sum_{i = 1}^m \log |\Theta_i| - \sum_{i = 1}^k \tr (\Theta B_i \Theta_i B_i^\dag). \label{eq:ipca-def}
\end{align}
Here $|X|$ denotes the determinant of $X$. The minimization ranges over $p\times p$ positive definite matrices $\Theta$ and $q_i\times q_i$ positive definite matrices $\Theta_i$. For comparison, the PCA axes are the eigenvectors of the matrix $\htheta$ obtained by restricting each $\Theta_i$ to be the $q_i \times q_i$ identity matrix in the above optimization problem. We call $\htheta, \htheta_1, \dots, \htheta_k$ the \emph{iPCA precision matrices}. Usually the main focus is in estimating $\Theta$; the remaining parameters are thought of as nuisance parameters.
%\CF{we may never need this notation - let's see} The eigenvectors of $\htheta_1, \dots, \htheta_k$ are called \emph{iPC loadings}, while the eigenvectors of $\htheta$ are called \emph{iPC scores}.

\begin{figure}
$$\begin{array}{ccccc}
 \underset{B_1}{\left[\begin{array}{ccc} | &  \dots & | \\ x_1 & \dots & x_{q_1} \\  | &  \dots & |  \end{array}\right]}, & \underset{B_2}{\left[\begin{array}{ccc} | &  \dots & | \\ x_{q_1 + 1} & \dots & x_{q_1 + q_2} \\  | &  \dots & |  \end{array}\right]}, & \dots &, \underset{B_k}{\left[\begin{array}{ccc} | &  \dots & | \\ x_{n - q_k + 1} & \dots & x_{n} \\  | &  \dots & |  \end{array}\right]} & \hspace{-.3cm}\left.\vphantom{\left[\begin{array}{ccc} | &  \dots & | \\ x_1 & \dots & x_{q_1} \\  | &  \dots & |  \end{array}\right]}\right\}p
\end{array}$$
\caption{Data points grouped as matrices.}\label{fig:groups}
\end{figure}

Like PCA, the iPCA precision matrices can be interpreted as a maximum likelihood estimator (MLE) for the precision matrix of a centered Gaussian random variable. Whereas the PCA axes are eigenvectors of the MLE $\htheta$ for the precision matrix assuming the data points are drawn independently from a centered Gaussian, the iPCA precisions $\htheta, \htheta_1, \dots, \htheta_k$ are the MLE for the precision matrix of a Gaussian with certain dependencies within each group. More precisely, we assume that each matrix $B_i$ is drawn independently from a centered Gaussian with precision matrix given by the Kronecker product $\Theta \ot \Theta_i$ for a $p\times p$ positive definite matrix $\Theta$ and a $q_i\times q_i$ positive definite matrix $\Theta_i$ (viewing $B_i$ as a vector in $\RR^{p \cdot q_i}$). Equivalently, $B_i$ is distributed as the $p\times q_i$ matrix $\Theta^{-1/2} Y \Theta_i^{-1/2}$ where $Y$ is a $p\times q_i$ matrix with independent, standard normal entries. Under this assumption, $\htheta \ot \htheta_i$ is the maximum likelihood estimator for $\Theta \ot \Theta_i$. 

\FloatBarrier

The maximum likelihood interpretation of iPCA captures several estimation problems of interest. The case $m = 1$ is the \emph{matrix normal model}, or the estimation of a covariance/precision matrix that can be written as a Kronecker product. The matrix normal model is used to model matrix-variate data that arises naturally in numerous applications like gene microarrays, spatio-temporal data, and brain imaging \cite{werner2008estimation}. On the other hand, when $q_i = 1$ for all $1 \leq i \leq m$, one finds that $\htheta$ is proportional to \emph{Tyler's M estimator} for the shape of an elliptical distribution \cite{tyler1987distribution}. iPCA interpolates between these two situations. We can interpret iPCA as an estimator for the precision matrix $\Theta$ that is robust to each data point in the group being replaced by a linear combination of the other data points in its group (i.e. replacing $B_i$ by $B_i \Theta_i^{-1/2}$) which can be chosen adversarially after seeing the data. As illustrated by Allen and Tang \cite{tang2018integrated}, such corruptions can obscure the true precision matrix from PCA and other methods for integrating groups of data. For a detailed discussion on the advantages and usefulness of iPCA and its strong performance in case studies, we refer the reader to \cite{tang2018integrated} where iPCA was used as a preprocessing method in a machine learning pipeline for predicting patients' cognition and Alzheimer's diagnosis.

%More precisely, iPCA recovers the axes that best explain the variance for the true data even under the assumption that the data the experimenter sees in each group is an unknown linear combination of the true data points within the group. 

\subsection{Existence and uniqueness of iPCA}
Despite the practical utility of iPCA, the existence and uniqueness of iPCA is poorly understood. Like PCA, iPCA need not exist if the data is too pathological. However, it is known that PCA exists for \emph{generic} data whenever there are as many samples as there are features. In contrast, it was not known in which dimensions and for which sizes of groups iPCA exists for generic data. In the extreme case when all the groups are of size one (Tyler's M estimator), it is known that iPCA generically exists whenever the number $n$ of data points is strictly larger than the dimension $d$ \cite{tyler1987distribution}. On the other extreme, the number of samples required for the MLE to generically exist uniquely for the matrix normal model was found exactly in \cite{derksen2020maximum}, and is on the order of $p/q + q/p$. In \cite{tang2018integrated}, Tang and Allen largely bypass the question of existence and uniqueness of iPCA by considering instead a regularized version of \cref{eq:ipca-def}.

In this paper we characterize the existence and uniqueness of iPCA. We say iPCA generically exists for dimension $p$ and group sizes $q_1, \dots, q_k$ if the iPCA precisions defined in \cref{eq:ipca-def} exist and are unique for generic $B_1, \dots, B_k$. In particular, if iPCA generically exists uniquely then the iPCA precisions exist with probability $1$ if $B_1, \dots, B_k$ is chosen according to any distribution on $\mat(p, q_1)\times\dots \times \mat(p, q_k)$ that has a density with respect to the Lebesgue measure. The phrases `iPCA generically exists uniquely' and `\cref{eq:ipca-def} is generically bounded/unbounded above' are defined similarly. Our first result is a simple sufficient condition for the generic existence and uniqueness of iPCA.

\begin{theorem}\label{thm:main} iPCA generically exists uniquely for $p, q_1, \dots, q_k$ provided 
$$ q_i \leq \frac{p}{2} \text{ and } 2 p \leq n .$$
\end{theorem}
As iPCA does not exist whenever $p>  n$ or whenever $q_i > p$ for some group $i$, the theorem is tight up to constant factors. When all the groups sizes are the same, we obtain the following necessary and sufficient conditions for the generic unique existence of iPCA. 
\begin{theorem}\label{thm:same-dim}
 Fix $k$. Let $\Gamma(p,q) = kq^2 + p^2 - kpq$. Let $r = {\rm gcd}(p,q)$. Then
\begin{itemize}
\item $\Gamma(p,q) \leq r^2$ implies iPCA generically exists for $(p, q, \dots, q)$.
\item $\Gamma(p,q)  < 0$ or $\Gamma(p,q)  \in \{0,r^2\}$ with $r = 1$ implies iPCA generically exists \emph{uniquely} for $(p, q, \dots, q)$.
\item In all other cases, i.e., $\Gamma(p,q)  > r^2$, \cref{eq:ipca-def} is generically unbounded above (iPCA does not exist).
\end{itemize}

\end{theorem}

When the group have differing sizes, the characterization is more complex. Nonetheless, well-known polyhedral characterizations of quiver stability imply a polyhedral characterization for the existence and uniqueness of iPCA. The conditions for existence and uniqueness will require checking firstly that the dimension vector resides in a certain cone, which will be a union of polyhedral cones, and secondly that certain divisibility conditions are satisfied. To this end, we define the \emph{region of generic unique existence} as the set 
$$S = \{p,q_1,\dots,q_k \in \R^{k + 1}:\text{iPCA generically exists uniquely for } p,q_1,\dots,q_k\} \subset \R^{k + 1} $$
and the \emph{cone of generic unique existence} as $\R_{\geq 0} S = \{a v: a \geq 0, v \in S\}$.

\begin{theorem}\label{thm:polyhedral}
Suppose $q_1 + \dots + q_k = n$. Let $d = n/p$.

\begin{enumerate}
    \item iPCA generically exists for $p,q_1,\dots,q_k$ if and only if \cref{eq:ipca-def} generically bounded above.
    
    \item There is a convex polyhedral cone $\Sigma_{d,k} \subset \RR^{k+1}$ such that iPCA exists for $p, q_1, \dots, q_k$ if and only if
$(p,q_1, \dots, q_k) \in \Sigma_{d,k}.$  

\item The cone of generic unique existence for iPCA is a union of finitely many convex polyhedral cones. Moreover, generic existence and uniqueness of iPCA for $p,q_1,\dots,q_k$ can be decided in $\poly(p,n, q_1, \dots, q_k)$.
\end{enumerate}

\end{theorem}
\begin{remark}[Time complexity]
As the algorithm in Theorem \ref{thm:polyhedral} depends polynomially on $q_1, \dots, q_k$ rather than the total length of their binary encodings, it is not technically a polynomial time algorithm. However, it \emph{is} polynomial time in the input size to the iPCA estimation problem which is a vector of dimension $n \cdot (\sum_i q_i)$. For this reason we abuse notation and refer to this algorithm as polynomial time in the remainder of this work.
\end{remark}

\FloatBarrier

The polyhedral characterization given by the theorem suggests that $p, q_1,\dots, q_k$ need not span a convex cone, but rather a union of convex cones which may or may not be convex. \cref{fig:nonconvex} shows a non-convex example.

Though the above theorem gives an efficient deterministic algorithm to decide generic existence and uniqueness of iPCA, this algorithm is rather complicated. We also provide a simple randomized algorithm to decide generic existence and uniqueness of iPCA. Computing the iPCA precisions is an instance of an algorithmic problem called \emph{operator scaling}, which is known to be solvable to an arbitrary degree of precision in polynomial time \cite{GGOW16}. 
%Our randomized algorithm does not use operator scaling, however, because the existence of arbitrarily precise approximate minimizers in \cref{eq:ipca-def} is not equivalent to the uniqueness or even existence of exact minimizers. \CF{well... it might use it to check semistability first.}

\begin{figure}
\includegraphics[width=.5\textwidth]{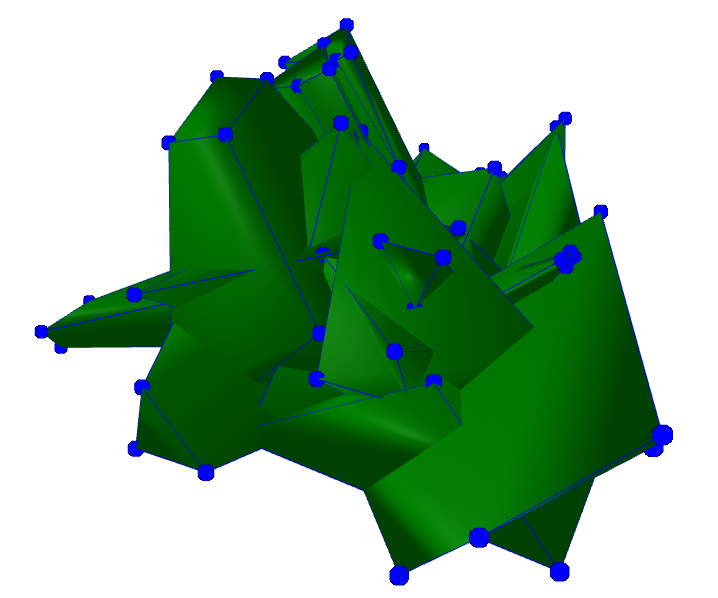}
\caption{The intersection of the cone of generic unique existence for $k=5$ in $\R^6$ with a random three-dimensional subspace through the origin. 
%\CF{include which slice for reproducibility?} 
}\label{fig:nonconvex}
\end{figure}

\begin{remark}[Consistency of iPCA] Interpreting iPCA as a maximum likelihood estimator, it is natural to ask under what circumstances the maximum likelihood estimator $\hat{\Theta}$ is likely to be close to the true $\Theta$. This question was studied in \cite{tang2018integrated} for regularized versions of iPCA. While this is not the focus of the present paper, the works \cite{franks2021near, franks2020rigorous} provided bounds for error rates for the matrix normal model and Tyler's M estimator, respectively, using tools from geodesically convex optimization. We conjecture that their techniques applied in this setting will yield error bounds comparable to those of PCA under slightly stronger hypotheses than \cref{thm:main}. Namely, we conjecture that for $p = \tilde{\Omega} (\max\{q_1,\dots, q_n\})$ and $n = \tilde{\Omega}(p)$ we have $\| I - \Theta^{-1/2} \hat{\Theta} \Theta^{-1/2}\|_{op} = \tilde{O}(\sqrt{p/n})$ with high probability in $n$. The $\tilde{\Omega},\tilde{\Omega}$ hide polylogarithmic factors in $q_i$ and $p$, respectively. 
%is consistent, namely if $(\hat{\Theta}, \hat{\Theta}_1, \dots, \hat{\Theta}_k) \to (\Theta, \Theta_1, \dots, \Theta_k)$ in probability.
\FloatBarrier
\end{remark}

\subsection{iPCA and stability for star quivers}

Our central observation is that the existence and uniqueness of iPCA can be interpreted in terms of \emph{quiver representations}. A quiver representation is a collection of vector spaces together with a collection of linear maps between the vector spaces. The spaces and directions of linear maps are represented by a directed multigraph. The iPCA problem arises from the so-called \emph{star quiver} (see the picture below). Given data $B_1 \in \mat_{p,q_1}, \dots, B_m \in \mat_{p,q_m}$, we form a quiver representation in a natural way. The vertices are the spaces $\RR^{p}, \RR^{q_1}, \dots, \RR^{q_m}$, and for $i \in [m]$ there is a directed edge from $\RR^{q_i}$ to $\RR^{p}$ corresponding to the linear map $B_i$:

\begin{equation}\begin{tikzcd}
& \RR^{p} & &  \\
\RR^{q_1} \arrow[ur, ->, "B_1"] &  \RR^{q_2}\arrow[u, ->,"B_2"] &  \dots & \RR^{q_k} \arrow[ull, ->,"B_k" above]
%\& \mu(X)_{2} \& \&
\end{tikzcd}\label{eq:quiver}\end{equation}

The existence and uniqueness of iPCA is captured by {\em stability} notions on the associated quiver representations. Stability is a technical notion from geometric invariant theory, whose definition we postpone to Section~\ref{sec:quiv}. A surprising and extremely useful link between maximum likelihood estimation and stability notions in invariant theory was discovered recently \cite{AKRS}, which we will recall in Section~\ref{sec:it-mle}.

Postponing the exact definitions, we indicate briefly a few things to give the reader an idea of the broad strokes of our approach and the techniques and ideas we draw from. For a quiver $Q$, a dimension vector $\alpha$ and a weight $\sigma$, a natural algorithmic question to ask is whether a generic $\alpha$-dimensional representation of $Q$ is $\sigma$-semistable/polystable/stable. In the case of iPCA, we consider the star quiver as above and take the dimension vector $\alpha = (p,q_1,\dots,q_k)$ as indicated above with the weight $\sigma = (-n,p,\dots,p)$. We find that that $\sigma$-semistability/polystability/stability in this setting is equivalent to having Eq~\ref{eq:ipca-def} generically bounded above/iPCA generically existing/iPCA generically existing uniquely. 

In addition to all the general results from quiver representations that we use, we wish to point out a few important features in this particular setting that are useful to us. With completely unrelated motivations, Schofield had identified a special weight (which we call the Schofield weight) for any dimension vector. The first crucial observation is that for the dimension vector $(p,q_1,\dots,q_k)$ for the star quiver, the Schofield weight is $\sigma = (-n,p,\dots,p)$, the precise weight relevant to iPCA. The second important idea is an overutilization of symmetry. While we naturally view the iPCA data as a representation of the star quiver, it can just as naturally be viewed as a representation of its opposite quiver (i.e., the star quiver with all arrows reversed). Well known to quiver theorists is a non-trivial way to relate the star quiver and its opposite through the use of reflection functors. Combining reflection functors with other obvious symmetries allows us to reduce the number of cases to consider significantly, in particular allows us to prove Theorem~\ref{thm:main}. The reader familiar with castling transforms and prehomogeneous spaces will appreciate the similarity to the classication of prehomogenous spaces of tensors, see e.g., \cite{venturelli2019prehomogeneous}. Finally, we note that generic semistability can be studied via semi-invariants. In the case of interest for iPCA, one can relate the semi-invariants to extremely well studied numbers in algebraic combinatorics called Littlewood-Richardson coefficients. In particular, Knutson and Tao's seminal work on the saturation conjecture and other results that followed provides us with the tools needed for us to exhibit algorithms that run in polynomial time in the node dimensions.

We also feel it is worth mentioning that while iPCA is naturally a problem over the field of real numbers, most known results on quiver representations and stability hold for algebraically closed fields such as the field of complex numbers. This creates a rather subtle difficulty in applying the results on quiver representations directly to iPCA. In fact, this is a problem that one is very likely to encounter while trying to relate stability or invariant theory to statistics in any setting in the manner discovered in \cite{AKRS}. To overcome this problem, one must appeal to some deep and technical results on algebraic groups. Hence, that discussion is kept entirely in the appendix for the sake of brevity and clarity.

\subsection{Related work}
The independent work \cite{chindris2022membership}, which was released during the preparation of this draft, also provides an algorithm for the generic existence and uniqueness of iPCA.

\subsection{Organization and notation} 

In \cref{sec:it-mle} we discuss the relationship between Gaussian group models and stability in invariant theory, and in particular how iPCA can expressed in this language. In \cref{sec:quiv} we review the notions we need about the stability of quivers. In \cref{sec:ipca-star} we relate the generic existence of iPCA to the Scofield stability of star quivers, and then in \cref{sec:scho-star} we prove results about Schofield stability and use them to finally prove Theorem~\ref{thm:main} and Theorem~\ref{thm:same-dim}. In \cref{sec:polytopal} we describe the polytopal characterization of the existence of iPCA and our algorithmic results, proving \cref{thm:polyhedral}.

We use the following symbols. The tuple $(p, q_1, \dots, q_k)$ will denote the dimensions in iPCA, and $n:=\sum_i q_i$.
We use $B_i:\RR^{q_i} \to \RR^{p}$ for the linear maps in iPCA, and $\Theta$, $\Theta_i$ for the solution to iPCA (inverses of what appear in Allen-Tang). Finally we use $\Sigma_{d,n}$ for the polyhedron describing semistability under the weight $(-n, p, \dots, p)$, or equivalently the weight $(-d,1,\dots,1)$.

%\begin{remark}[Use for moment map]\CF{Actually, I think if we are trying to think of iPCA as an analogy of PCA we should be perfectly fine with situations where the MLE doesn't exist. For instance, in PCA all the data can be on a lower dimensional subspace. PCA handles this by outputting a low rank matrix. I think we should be ok with doing something similar for iPCA, I am just not sure how to make it happen. Na\"ively, just scaling to something lower rank (on the sink) seems like it might work.}\end{remark}

%\section{iPCA and the stability of the simply laced star quiver}
%We begin by translating statements about iPCA into notions of quiver stability, which we now define. 
%\begin{definition}[Quiver stability, semistability, and instability]
%\CF{todo}
%\end{definition}
%We use the notion of genericity.
%\begin{definition}[Generic] \CF{todo}

%\end{definition}

%\begin{theorem}Let $Q$ be the quiver described in \cref{eq:quiver}.
%\begin{enumerate}
%\item iPCA generically exists for $d, d_1, \dots, d_m$ if and only if $Q$ is generically semistable.
%\item iPCA generically exists \emph{uniquely} for $d, d_1, \dots, d_m$ if and only if $Q$ is generically stable.
%\item \cref{eq:ipca-def} is unbounded below for $d, d_1, \dots, d_m$ if and only if $Q$ is generically stable.
%\end{enumerate}
%\end{theorem}
%\begin{proof}\CF{todo}

%\end{proof}

%\subsection{Proof of \cref{thm:main}}

%\subsection{Proof of \cref{thm:same-dim}}

%\subsection{Proof of \cref{thm:polyhedral}}

%\subsection{A simple randomized algorithm} 

\section{Invariant theory and MLE for Gaussian group models} \label{sec:it-mle} 
Invariant theory is the study of symmetries captured by group actions. From its very beginnings, there has been significant focus on computation in the subject of invariant theory and as a consequence, invariant theory finds many applications. We refer to \cite{DK} and references therein for a comprehensive introduction to the subject and its applications. In this section, we will recall some basic notions in invariant theory and draw the connection to maximum likelihood estimation that was discovered recently in \cite{AKRS}. 

For this section, let $K = \R$ or $\C$ be the ground field. The basic setting is as follows. A representation of a group $G$ is an action of $G$ on a (finite-dimensional) vector space $V$ (over the field $K$) by linear transformations. Equivalently, a representation can be thought of as a group homomorphism $\rho: G \rightarrow \GL(V)$. In particular, an element $g \in G$ acts on $V$ by the linear transformation $\rho(g)$. We write $g \cdot v$ or $gv$ to mean $\rho(g)v$. Throughout this paper, we will only consider the setting where $G$ is a linear algebraic group (over the ground field $K$), i.e., $G$ is an (affine) variety, the multiplication and inverse maps are morphism of varieties, and the action is a rational action (or rational representation), i.e., $\rho: G \rightarrow \GL(V)$ is a morphism of algebraic groups.

The $G$-orbit of $v \in V$ is
$$
O_v := \{gv\ |\ g \in G\} \subseteq V,
$$
and we denote by $\overline{O_v}$ the closure of the orbit $O_v$. 

\begin{remark} \label{topology}
To define the closure, we need to define a topology on $V$. In this paper, we will only use the fields $K = \R$ or $\C$ and so we use the standard Euclidean topology on $V$ for orbit closures, unless otherwise specified. This is not standard. In literature, the topology is usually taken as the Zariski topology. Thankfully, in the setting of rational actions of reductive groups, for $K = \C$, the orbit closure with respect to the Euclidean topology agrees with the orbit closure with respect to the Zariski topology. 
%We will need to use the Zariski topology at times, but we will be careful in specifying it each time. 

%We caution the reader that the interplay between the Euclidean and Zariski topology can be a bit tricky at times for $K = \R$.
\end{remark}

%We denote by $K[V]$, the ring of polynomial functions on $V$ (a.k.a. the coordinate ring of $V$). A polynomial function $f \in K[V]$ is called {\em invariant} if $f(gv) = f(v)$ for all $g \in G$ and $v \in V$. In other words, a polynomial is called invariant if it is constant along orbits. The invariant ring is 
%$$
%K[V]^G := \{f \in K[V]\ |\ f(gv) = f(v) \ \forall\ g \in G, v \in V\}.
%$$

%The invariant ring has a natural grading by degree, i.e., $K[V]^G = \oplus_{d=0}^\infty K[V]^G_d$ where $K[V]^G_d$ consists of all invariant polynomials that are homogeneous of degree $d$. For $v \in V$, 

%Note that any invariant polynomial is constant along an orbit and hence constant along an orbit closure because polynomials are continuous (whether we consider Zariski or Euclidean topology). 

For a point $v \in V$, the subgroup $G_v := \{g \in G \ |\ gv = v\}$ is called the stabilizer. We now define stability notions in invariant theory that play a central role in this paper.

\begin{definition}
Let $K = \R$ or $\C$, and let $G$ be an algebraic group (over $K$) with a rational action on a vector space $V$ (over $K$), i.e., $\rho: G \rightarrow \GL(V)$. Let $\Delta$ denote the kernel of the homomorphism $\rho$. Give $V$ the standard Euclidean topology. Then, for $v \in V$, we say $v$ is 
\begin{itemize}
\item {\em unstable} if $0 \in \overline{O_v}$;
\item  {\em semistable} if $0 \notin \overline{O_v}$;
\item  {\em polystable} if $v \neq 0$ and $O_v$ is closed;
\item  {\em stable} if $v$ is polystable and the quotient $G_v/\Delta$ is finite.
\end{itemize}
\end{definition}

%We point out again that our definitions may not be quite standard because we use the Euclidean topology. However, this is the form that is most suited for our purposes. 

Observe that a point is unstable if and only if it is not semistable and also note that stable $\implies$ polystable $\implies$ semistable. 

For any action of $G$ on $V$, there is a natural diagonal action on the direct sum $V^m$ defined by $g\cdot (v_1,\dots,v_m) = (gv_1,\dots,gv_m)$ for all $g \in G$ and $v_i \in V$. Moreover, note that for any group action $\rho: G\rightarrow \GL(V)$, the notions of semistable, polystable and stable are the same whether we consider the action of $G$ or the action of the group $\rho(G)$ on $V$ (by matrix-vector multiplication).

The set of all unstable points is called the {\em null cone}, a central object in the study of computational invariant theory. When $K = \C$, the invariant ring is non-trivial precisely when a generic point is semistable, which is equivalent to the existence of a semistable point, i.e., the null cone is not all of $V$. The null cone membership is the problem of deciding whether a vector is in the null cone, i.e., whether it is unstable. This problem has received enormous amount of attention in recent years, see \cite{BFGOWW} (and references therein) for a comprehensive overiew. Similarly, the problem of deciding whether a given vector is polystable has also been studied due to its applications to degree lower bounds for invariant rings, see \cite{DM-exp, DM}. %\CF{visu get the proper citation here}

%\VM{In this subsection itself, need to say the set of unstable vectors is the null cone (for $K = \C$) and that invariant ring is non-trivial if and only if there exist semistable points if and only if generic point is semistable}

\subsection{MLE for Gaussian group models and invariant theory} \label{sec:gaussianmodel}
A curious and unexpected connection between stability notions defined above and maximum likelihood estimation was discovered in \cite{AKRS}. We now explain this connection in more detail and in the subsequent subsection, we describe explicitly the relation to iPCA.

Suppose $K = \R$ or $\C$. We consider multivariate Gaussian models on $K^n$. In the case of $K = \R$, we take our Gaussian distributions to be centered (i.e. mean zero), so that the distribution is defined by simply giving its precision matrix which is a positive definite matrix. In the case of $K = \C$, we take our Gaussian distributions to be circularly symmetric, so that the distribution is again defined by simply giving its precision matrix which is a positive definite Hermitian matrix. Let $\PD_n$ denote the space of positive definite (Hermitian when $K = \C$) matrices with entries in $K$. By the above discussion, a subset of $\PD_n$ defines a Gaussian model.

Let $\rho: G \rightarrow \GL_n$ be a representation of a group $G$. Then, we define a corresponding {\em Gaussian group model} $\mathcal{M}_G := \{\rho(g)^\dag  \rho(g)\ |\ g \in G\} \subseteq \PD_n$.\footnote{Note that adjoint is the same as transpose for a matrix with real entries.} Informally speaking, the maximum likelihood estimation problem decomposes into two parts. Firstly, estimate the maximum likelihood up to a scalar factor by computing the maximum likelihood estimator in a subgroup of matrices of determinant $1$, denoted $G_{SL}$. After this is done, finding the correct scalar factor becomes a trivial one-dimensional optimization problem with a closed form solution. This approach leads to the following result for Gaussian group models proved in \cite{AKRS} (we state a more general, but equivalent form of their result; see also \cite{derksen2020maximum}). 

\begin{theorem} [\cite{AKRS}] \label{theo:AKRS}
Let $K = \R$ or $\C$, and let $V$ be a finite dimensional Hilbert space, i.e., a vector space with a positive definite inner product (Hermitian when $K = \C$). Let $\rho:G \rightarrow \GL(V)$ be a rational action of $G$ on $V$.  Suppose $\rho(G) \subseteq \GL(V)$ is a Zariski closed subgroup, closed under adjoints and non-zero scalar multiples. Let $G_{\SL} \subseteq G$ be a subgroup such that $\rho(G_{\SL}) = \rho(G) \cap (\SL(V))$ and let $Y \in V^m$ be an $m$-tuple of samples. Then, for the (diagonal) action of $G_{\SL}$ and the model $\mathcal{M}_G$, we have
\begin{itemize}
\item  $Y$ is semistable $\Longleftrightarrow$ The likelihood $l_Y$ is bounded from above; 
\item $Y$ is polystable $\Longleftrightarrow$ an MLE exists;
\item $Y$ is stable $\implies$ there is a unique MLE. If $K = \C$, the converse also holds, i.e., there is a unique MLE $\implies Y$ is stable.
\end{itemize}
\end{theorem}

The fact that the converse of the last part does not hold for $K = \R$ is a bit disappointing. However, a partial result in that regard which is very useful is the following:

\begin{lemma} \label{lem:rect-AKRS}
Keep the notation of Theorem~\ref{theo:AKRS}. Let $K = \R$. If there is a unique MLE given $Y$, then the image of the stabilizer $(\rho(G_{\SL}))_Y = \rho((G_{\SL})_Y)$ is compact.
\end{lemma}

\begin{proof}
Suppose there is a unique MLE given $Y$. Suppose $g \cdot Y$ is a point of minimal norm in the $G_{\SL}$-orbit of $Y$. Then, the MLEs are in bijection with $\rho((G_{\SL})_{g \cdot Y})/\rho((G_{\SL})_{g \cdot Y}) \cap O(V)$, where $O(V)$ is the subgroup of $V$ that preserves the inner product, see \cite{AKRS}. In particular, we have a unique MLE if and only if $\rho((G_{\SL})_{g \cdot Y}) \subseteq O(V)$. Since $\rho((G_{\SL})_{g \cdot Y}) = (\rho(G_{\SL}))_{g \cdot Y}$ is a closed subgroup of $O(V)$ (indeed stabilizers are closed subgroups), we see that it is compact because $O(V)$ is compact. As $\rho(G_{\SL})_{g \cdot Y}$ is conjugate to $\rho(G_{\SL})_Y$, we get that the latter is compact as well.
\end{proof}

%\VM{We need to say $\rho(G_{\SL})_Y$ is compact because if $\rho$ has a non-compact kernel, then the stabilizer will never be compact.}

%\begin{remark}
%In the above result, it suffices to ask for $\rho(G_{\SL})$ and $\rho(G) \cap (\SL(V))$ to have the same identity component since the stability notions only depend in the identity component of the group acting. This mild generalization will be important for us.
%\end{remark}

\subsection{iPCA as a Gaussian group model} \label{sec:ipca-gaussianmodel}
In this subsection, we place iPCA in the framework of Gaussian group models so that we can study existence and uniqueness of iPCA by studying the corresponding invariant theory. Later on, we will see that the invariant theory that is related to iPCA is the invariant theory of quivers, in particular the star quiver, see Section~\ref{sec:ipca-star}.

Let $K = \R$ or $\C$. Let $\alpha = (p,q_1,\dots,q_k)$ and recall that $n = \sum_i q_i$. Consider the group $G = \GL(\alpha) = \GL_p \times \prod_{i=1}^k \GL_{q_i}$. Consider the action $\rho = \rho_\alpha$ of $G$ on $\oplus_{i=1}^k \mat_{p,q_i}$ given by
\begin{equation} \label{eq:ipca-gp-action}
(g,h_1,\dots,h_k) \cdot (B_1,\dots,B_k) = (g B_1 h_1^{-1}, gB_2 h_2^{-1}, \dots, gB_k h_k^{-1}).
\end{equation}

Identify $\oplus_{i=1}^k \mat_{p,q_i}$ with $K^{pn}$ as follows. For matrices $(B_1,\dots,B_k) \in \oplus_{i=1}^k \mat_{p,q_i}$, place them (horizontally) next to each other to get a matrix of size $p \times n$. Then, vectorize this matrix, i.e., stack its columns. With this identification, the representation $\rho: G \rightarrow \GL_{pn}$ given by 
$$
\rho(g,h_1,\dots,h_k) = g \otimes {\rm diag}((h_1^{-1})^\top,\dots,(h_k^{-1})^\top),
$$

where $\otimes$ denotes the Kronecker (or tensor) product of matrices and ${\rm diag}((h_1^{-1})^\top,\dots,(h_k^{-1})^\top)$ denotes a block diagonal matrix whose diagonal blocks are $(h_1^{-1})^\top,\dots, (h_k^{-1})^\top$.

For a choice of $\sigma = (\sigma_0,\sigma_1,\dots,\sigma_k) \in \Z^{k+1}$, we define the subgroup 
$$
\GL(\alpha)_\sigma := \{(g,h_1,\dots,h_k) \in \GL(\alpha)\ |\ \det(g)^{\sigma_0} \cdot \prod_{i=1}^k \det(h_i)^{\sigma_i} = 1\}.
$$

\begin{lemma} \label{lem:GSL-Gsigma}
Let $\sigma = (-n, p, p,\dots, p)$. Then $\rho_\alpha(G) \cap \SL_{pq} = \rho_\alpha(\GL(\alpha)_\sigma)$.
\end{lemma}

\begin{proof}
This is clear since $\det(\rho_\alpha(g)) = \det(g)^{n} \cdot \prod_{i=1}^k \det(h_i)^{-p}$.
\end{proof}

%\begin{lemma}
%Let $d = {\rm gcd}(p,n)$ and let $p' = p/d$ and $n' = n/d$. Let $\mu = (-n',p',p',\dots,p')$. Then $\rho(G) \cap \SL_{pq}$ and $\GL(\alpha)_\mu$ have the same identity component.
%\end{lemma}

%\begin{proof}
%Let $\sigma = (-n, p, p,\dots, p)$. Then $\GL(\alpha)_\sigma \supseteq \GL(\alpha)_\mu$. Moreover, $\GL(\alpha)_\mu$ is clearly the kernel of the map $\pi: \GL(\alpha)_\sigma \rightarrow K$ given by $(g,h_1,\dots,h_k) \mapsto \det(g)^{-n'} \cdot \prod_{i=1}^k \det(h_i)^{p'}$. Since $\pi^d (g) = 1$ for all $g \in \GL(\alpha)_\sigma$, the image $\pi(\GL(\alpha))_\sigma$ is contained in the discrete subset consisting of $d^{th}$ roots of unity. Thus, $\GL(\alpha)_\mu$ is a subgroup of finite index in $\GL(\alpha)_\sigma$ and hence has the same identity component as $\GL(\alpha)_\sigma$.
%\end{proof}

\begin{proposition} \label{prop:ipca-to-gaussianmodel}
Let $K = \R$ or $\C$. Let $\alpha = (p,q_1,\dots,q_k)$, $G = \GL(\alpha)$ and consider the action $\rho_\alpha$ of $G$ on $\Rep(\Delta[k], \alpha)) = \oplus_{i=1}^k \mat_{p,q_i}$. Let $\mathcal{M}_\alpha$ denote the corresponding Gaussian group model. Let $\sigma = (-n, p, p ,\dots, p)$. Then, for $B \in \Rep(\Delta[k], \alpha))$, we have
\begin{itemize}
\item  $B$ is $\GL(\alpha)_\sigma$-semistable $\Longleftrightarrow l_B$ is bounded from above; 
\item $B$ is $\GL(\alpha)_\sigma$-polystable $\Longleftrightarrow$ an MLE exists;
\item $B$ is $\GL(\alpha)_\sigma$-stable $\implies$ there is a unique MLE. Further, if $K = \C$, the converse also holds, i.e., there is a unique MLE $\implies B$ is $\GL(\alpha)_\sigma$-stable. Moreover, when $K = \R$, there is a unique MLE $\implies$ the stabilizer $(\rho_\alpha(\GL(\alpha)_\sigma))_B$ is compact.
\end{itemize}
\end{proposition}

The notation $\Rep(\Delta[k], \alpha))$ is certainly uninitiated but will become clear when we discuss quivers. But we use this notation to maintain consistency in the paper.

\begin{proof}
This follows from Theorem~\ref{theo:AKRS} and Lemma~\ref{lem:rect-AKRS} since we can take $\GL(\alpha)_\sigma$ as $G_{\SL}$ by Lemma~\ref{lem:GSL-Gsigma}.
\end{proof}

Now, a simple computation shows that the model $\mathcal{M}_\alpha$ is given by
\begin{equation} \label{eq:ipca-model}
\mathcal{M}_\alpha = \{\Theta \otimes {\rm diag}(\Theta_1,\dots,\Theta_k)\ |\ \Theta \in \PD_p, \Theta_i \in \PD_{q_i}\}.
\end{equation}
Allen and Tang \cite{tang2018integrated} describe iPCA as the maximum likelihood estimation for the  Gaussian model $\mathcal{M}_\alpha$. Moreover, in \cite{tang2018integrated}, the  log-likelihood function of $\mathcal{M}_\alpha$ is computed.

%\VM{@CF: Why did you remove the proof of this and cite Allen-Tang instead? In my mind, the point was that the log-likelihood of the Gaussian group model is the same as the function whose optima give iPCA precision matrices, and therefore the next corollary}
%\CF{I think Allen-Tang describe iPCA as a Gaussian group model too, without mentioning that language.}

Let $K = \R$. Let $\alpha = (p,q_1,\dots,q_k)$ and let $n = \sum_i q_i$. Let $\GL(\alpha), \rho_\alpha, \Rep(\Delta[k], \alpha))$ and $\mathcal{M}_\alpha$ be as above. Then, for $B = (B_1,\dots,B_k) \in \Rep(\Delta[k], \alpha))$, the log-likelihood function (see [Eq. 12, \cite{tang2018integrated}] is
$$
l_B (\Theta \otimes {\rm diag}(\Theta_1,\dots,\Theta_k)) =  n \log|\Theta| + p \sum_{i = 1}^k \log |\Theta_i| - \sum_{i=1}^k \tr (\Theta B_i \Theta_i B_i^\dag)
$$
There is a formula for the log-likelihood function of a Gaussian group model is given in \cite{AKRS}. The interested reader can verify that specializing that to the Gaussian group model $\mathcal{M}_\alpha$ you recover the same formula.

%\begin{proof}
%The log-likelihood of a Gaussian on a sample $v \in V_\alpha$ with precision matrix $\tilde{\Theta}:V_\alpha \to V_\alpha$ is given by 
%$$ \frac{1}{D}\sum_{i = 1}^n v^\dagger \Theta v -  \frac{1}{D}\log\det\Theta.$$
%The formula for the log-likelihood function of a Gaussian group model is given in \cite{AKRS}: 
%$$ log like$$
% Specializing to our specific situation, one easily recovers the above formula.
%\VM{Either provide a more precise reference or give a more detailed proof}
%\end{proof}

For sample data $B = (B_1,\dots,B_k) \in \Rep(\Delta[k], \alpha))$,  the iPCA precisions are equal to maximum likelihood estimators for $B$ with respect to the model $\mathcal{M}_\alpha$. More precisely, the function Eq~\ref{eq:ipca-def} of which the iPCA precisions are the optima is the log-likelihood function of $\mathcal{M}_\alpha$ and hence the iPCA precision matrices correspond to MLEs. Hence, existence and uniqueness of iPCA is the same as existence and uniqueness of MLE for the model $\mathcal{M}_\alpha$.

Summarizing the above discussion, we get:

\begin{corollary} \label{cor:ipca-Malpha}
Let $K = \R$. Let $\alpha = (p,q_1,\dots,q_k)$. Let $\GL(\alpha), \rho_\alpha, \Rep(\Delta[k], \alpha))$ and $\mathcal{M}_\alpha$ be as above. Then, for sample data in $\Rep(\Delta[k], \alpha))$:
\begin{itemize}
\item Eq~\ref{eq:ipca-def} is the log-likelihood function for $\mathcal{M}_\alpha$ is generically bounded above. 
\item iPCA precision matrices are precisely the MLE w.r.t to $\mathcal{M}_\alpha$.
\end{itemize}
\end{corollary}

\begin{remark}
If we take $K = \C$, then maximum likelihood estimation for the complex model $\mathcal{M}_\alpha$ can be seen as a complex version of iPCA. Our results on the real model (on existence and uniqueness of MLE) also hold for this complex model. Indeed, our proof strategy is to find the answers for the complex model and then transfer the results to the real model.
\end{remark}

\section{Quiver representations and stability}\label{sec:quiv}
From a purely algebraic standpoint, the theory of quiver representations was built to be a natural and rich generalization of linear algebra. Quiver representations have numerous applications to various areas in algebra and geometry, e.g., cluster algebras \cite{derksen2010quivers, Keller}, Schubert calculus \cite{DSW, Ressayre,DW-LR}, non-commutative algebraic geometry \cite{Reineke} and symplectic resolutions \cite{Ginzburg} to name a few. In the last decade, connections between quivers and computational complexity were discovered \cite{HW14} and this led to progress on non-commutative rational identity testing and more, see \cite{GGOW16, DM, IQS2, DM-arbchar, DM-si, GCTV, GGOW-BL}. With regard to statistics, exact sample size thresholds for matrix normal models were obtained in \cite{derksen2020maximum} by utilizing a connection to the invariant theory of Kronecker quivers (via Theorem~\ref{theo:AKRS}) that was first explained in \cite{AKRS}.

In this section, we will discuss the basic notions of quiver representations and stability of quiver representations. Then, we will explain in detail the connection between iPCA and star quivers, which will be used in the subsequent sections to prove our main results on existence and uniqueness of iPCA.

Let $K$ denote the ground field, we will only consider $K = \R$ or $\C$ in this paper. A quiver $Q$ is a directed acyclic graph\footnote{In literature, quivers are not always taken to be directed acyclic. But for our purposes, we will make this a standing assumption throughout not only because this is the case that is relevant to us, but also because the theory is better understood in this case.}, i.e. a set of vertices denoted $Q_0$ and a set of arrows $Q_1$. For each arrow $a \in Q_1$, we denote by $ta$ and $ha$, the tail vertex and head vertex of the arrow. We will demonstrate all the basic notions and definitions in the crucial example (below) of the star quiver $\Delta[k]$ with vertices $x$ and $y_1,\dots,y_k$ with $k$ arrows $a_1,\dots,a_k$ such that the arrow $a_i$ goes from $y_i$ to $x$ .

\begin{equation}\begin{tikzcd}
& x & &  \\
y_1 \arrow[ur, ->, "a_1"] &  y_2\arrow[u, ->,"a_2"] &  \dots & y_n \arrow[ull, ->,"a_k" above]
%\& \mu(X)_{2} \& \&
\end{tikzcd}\label{eq:star-quiver}\end{equation}

A representation $V$ of $Q$ is simply an assignment of a finite-dimensional vector space $V(x)$ for each $x \in Q_0$ and a linear transformation $V(a): V(ta) \rightarrow V(ha)$ for each arrow $a \in Q_1$.  A morphism of quiver representations $\phi: V \rightarrow W$ is a collection of linear maps $\phi(x): V(x) \rightarrow W(x)$ for each $x \in Q_0$ subject to the condition that for every $a \in Q_1$, the diagram below commutes.

\begin{center}
\begin{tikzcd}
V(ta) \arrow[r, "V(a)"] \arrow[d,, "\phi(ta)"]
& V(ha) \arrow[d, "\phi_{ha}"] \\
W(ta) \arrow[r, "W(a)"]
& W(ha)
\end{tikzcd}
\end{center}

A representation $V$ of $\Delta[k]$ is given by assigning vector spaces $V(x)$ to $x$ and and $V(y_i)$ to $y_i$ for each $i$, and $n$ linear maps $V(a_1),\dots,V(a_k)$ where $V(a_i): V(y_i)\rightarrow V(x)$. A morphism between two representations $V$ and $W$ of $\Delta[k]$ is given by linear maps $\phi(x): V(x) \rightarrow W(x)$ and $\phi(y_i): V(y_i) \rightarrow W(y_i)$ for all $i$ such that $\phi(x) \circ V(a_i) = W(a_i) \circ \phi(y_i)$ for all $i$.

A subrepresentation $U$ of $V$ is a collection of subspaces $U(x) \subseteq V(x)$ such that for every arrow $a \in Q_1$ the linear map $U(a)$ is simply a restriction of $V(a)$. In particular, this means that the image of $U(ta)$ under $V(a)$ will need to be contained in $U(ha)$. For two representations $V$ and $W$, we define their direct sum $V \oplus W$ to be the representation that assigns $V(x) \oplus W(x)$ to each vertex $x$ and the linear map $\begin{pmatrix} V(a) & 0 \\ 0 & W(a) \end{pmatrix}$ for each arrow $a \in Q_1$. Similarly, the notion of direct summand, image, kernel, co-image, etc. are all defined in a straightforward way. We say $V$ is indecomposable if it cannot be written as a direct sum of two (proper) subrepresentations. Otherwise, $V$ is called decomposable. In summary, the category of quiver representations forms an Abelian category. We refer to \cite{DW-book} for complete details.

The dimension vector of a representation $V$ is $\underline{\dim}(V) = (\dim V(x))_{x \in Q_0}$. So, for a representation $V$ of $\Delta[k]$, its dimension vector is 
$$\underline{\dim}(V) = (\dim(V(x)), \dim(V(y_1)), \dim(V(y_2)), \dots, \dim(V(y_k))).$$ We will keep the convention throughout this paper that when we write a dimension vector of $\Delta[k]$, we will first specify the dimension at $x$ and then the dimensions at $y_1,y_2,\dots,y_k$ in order.

For any representation $V$ of a quiver $Q$, for each $x \in Q_0$, picking a basis for $V(x)$ identifies $V(x)$ with $K^{\dim(V(x))}$. Further, with this identification, every linear map $V(a)$ is just a matrix of size $\dim(V(ha)) \times \dim(V(ta))$. Thus, we come to the following definition. For any dimension vector $\alpha = (\alpha(x))_{x \in Q_0} \in \N^{Q_0}$ (where $\N = \{0,1,2,\dots,\}$), we define the representation space
$$
\Rep(Q,\alpha) = \bigoplus_{a \in Q_1} \Mat_{\alpha(ha),\alpha(ta)}.
$$

Any point $V= (V(a))_{a \in Q_1} \in \Rep(Q,\alpha)$ can be interpreted as a representation of $Q$ with dimension vector $\alpha$ as follows: for each $x \in Q_0$, assign the vector space $K^{\alpha(x)}$, and for each arrow $a \in Q_1$, the matrix $V(a)$ describes a linear transformation from $K^{\alpha(ta)}$ to $K^{\alpha(ha)}$. The base change group $\GL(\alpha) = \prod_{x \in Q_0} \GL_{\alpha(x)}$ acts on $\Rep(Q,\alpha)$ in a natural fashion where $\GL_{\alpha(x)}$ acts on the vector space $K^{\alpha(x)}$ assigned to vertex $x$ by changing basis. More concretely, for $g = (g_x)_{x \in Q_0} \in \GL(\alpha)$ and $V = (V(a))_{a \in Q_1} \in \Rep(Q,\alpha)$, the point $g \cdot V \in \Rep(Q,\alpha)$ is defined by the formula
$$
(g \cdot V) (a) = g_{ha} V(a) g_{ta}^{-1}.
$$

The $\GL(\alpha)$ orbits in $\Rep(Q,\alpha)$ are in $1-1$ correspondence with isomorphism classes of $\alpha$-dimensional representations.

For the star quiver $\Delta[k]$, suppose we pick a dimension vector $\alpha = (p,q_1,\dots,q_k)$, then the representation space 
$$
\Rep(\Delta[k], \alpha)) = \bigoplus_{i=1}^n \Mat_{p,q_i}.
$$
%\CF{should we unify the notation $V_\alpha = \Rep(\Delta[k], \alpha))$? Also, a little confusing that $V$ means vector space and point in rep.}
%\VM{Okay, now I understand the issue you were referring to}
Now, $\GL(\alpha) = \GL_p \times \GL_{q_1} \times \GL_{q_2} \times \dots \times \GL(q_n)$, and the action is given by the formula
$$
(g_x, g_{y_1},\dots,g_{y_n}) \cdot (B_1,\dots,B_n) = (g_x B_1g_{y_1}^{-1},\dots, g_x B_n g_{y_n}^{-1}),
$$
which when compared with Eq~\ref{eq:ipca-gp-action} suggests the connection between iPCA and the star quiver.

%The orbits of this action correspond to isomorphism classes of $(p,q)$-dimensional representations of $\Theta(m)$. The subgroup $\SL(\alpha) = \SL_p \times \SL_q$. First, observe that $Y = (Y_1,\dots,Y_m)$ is semistable/polystable/stable (for the action of $\SL_p \times \SL_q$) if and only if $\lambda Y = (\lambda Y_1,\dots, \lambda Y_m)$ is semistable/polystable/stable for $\lambda \in K^\times$. This is a simple consequence of the fact that the action is by linear transformations. Thus, we see that whether $Y = (Y_1,\dots,Y_m)$ is semistable, polystable, or stable (for the action of $\SL(\alpha)$) only depends on the isomorphism class of the quiver representation it defines (i.e., the $\GL(\alpha)$-orbit). This is the starting point of understanding the various stability notions from a representation theoretic perspective, which we will discuss in more detail in the next section. 

\begin{remark}
The space $\Rep(Q,\alpha)$ is a representation of $\GL(\alpha)$ and its various subgroups. At the same time, we refer to a point $V \in \Rep(Q,\alpha)$ also as a representation. We advise the reader to keep in mind that we think of $V$ as a representation of the quiver $Q$ and not of any group to avoid confusion. Moreover, if $V \in \Rep(Q,\alpha)_\R = \bigoplus_{a \in Q_1} \Mat_{\alpha(ha),\alpha(ta)}(\R)$, then it can be thought of as both a real and complex representation of $Q$.
\end{remark}

\subsection{Stability notions}
We follow the conventions from \cite{DW-book} for consistency. For this section, we fix $K =\C$. We intentionally refrain from defining stability notions in the case of $K = \R$ to avoid confusion.

Let $Q$ be a quiver with no oriented cycles (self loops are counted as oriented cycles). Let $\alpha$ be a dimension vector. Consider the subgroup $\SL(\alpha) = \prod_{x \in Q_0} \SL(\alpha(x)) \subseteq \GL(\alpha)$. Then, the invariant ring for the action of $\SL(\alpha)$ on $\Rep(Q,\alpha)$ is called the ring of semi-invariants
$$
\SI(Q,\alpha) = K[\Rep(Q,\alpha)]^{\SL(\alpha)}.
$$

For any $\sigma = (\sigma(x))_{x \in Q_0} \in \Z^{Q_0}$ (which we call a weight), we have a character of $\GL(\alpha)$ which we also denote $\sigma$ by abuse of notation. The character $\sigma: \GL(\alpha) \rightarrow K^\times$ is given by $\sigma((g_x)_{x \in Q_0}) = \prod_{x \in Q_0} \det(g_x)^{\sigma(x)}$. The ring of semi-invariants has a decomposition

$$
{\rm SI}(Q,\alpha) = \bigoplus_{\sigma \in \Z^{Q_0}} {\rm SI}(Q,\alpha)_\sigma,
$$
where $\SI(Q,\alpha)_\sigma = \{f \in {\rm SI}(Q,\alpha) \ |\ f(g \cdot x) = \sigma(g^{-1}) f(x)\  \forall g \in \GL(\alpha)\}$. 

 We define the \emph{effective cone of weights} 
$$
\Sigma(Q,\alpha) := \{\sigma \in \Z^{Q_0}\ |\ {\rm SI}(Q,\alpha)_{m\sigma} \neq 0\ \text{for some } m \in \Z_{>0}\}.
$$
For a weight $\sigma$ and a dimension vector $\beta$, we define $\sigma(\beta) := \sum_{x \in Q_0} \sigma(x) \beta(x)$. We point out here that every $\sigma \in \Sigma(Q,\alpha)$ must satisfy $\sigma(\alpha) = 0$. For each $0 \neq \sigma \in \Sigma(Q,\alpha)$, we consider the subring
$$
\SI(Q,\alpha,\sigma) := \oplus_{m=0}^\infty \SI(Q,\alpha)_{m\sigma}.
$$

For a sincere dimension vector $\alpha$ (i.e., $\alpha(x) \neq 0 \ \forall x \in Q_0$), it turns out that this subring can also be seen as an invariant ring, i.e., $\SI(Q,\alpha,\sigma) = K[\Rep(Q,\alpha)]^{\GL(\alpha)_\sigma}$ where $\GL(\alpha)_\sigma = \{g \in \GL(\alpha)\ |\ \sigma(g) = 1\}$. Note that $\GL(\alpha)_\sigma$ is a reductive group. It is well-known that the associated projective variety ${\rm Proj}(\SI(Q,\alpha,\sigma))$ defines a moduli space for the $\alpha$-dimensional representations of $Q$, see \cite{King}.

We make a definition following King \cite{King}. We follow the convention from \cite{DW-book} which is consistent with our notational choices so far, but differs from King's original convention by a sign.

\begin{definition} [King \cite{King}] \label{crit-king}
Let $Q$ be a quiver with no oriented cycles, $V$ be a representation of $Q$ and $\sigma \in \Z^{Q_0}$ a weight such that $\sigma(\underline\dim V) = 0$.
\begin{itemize}
\item $V$ is $\sigma$-semistable if $\sigma(\beta) \leq 0$ for all $\beta \in \Z_{\geq 0}^{Q_0}$ such that $V$ contains a subrepresentation of dimension $\beta$.
\item $V$ is $\sigma$-stable if $\sigma(\beta) < 0$ for all $\beta \in \Z_{\geq 0}^{Q_0}$ (other than $0$ and $\underline{\dim}(V)$) such that $V$ contains a subrepresentation of dimension $\beta$.
\item $V$ is $\sigma$-polystable if $V = V_1 \oplus V_2 \oplus \dots \oplus V_k$ such that $V_i$ are all $\sigma$-stable representations.
\end{itemize}
\end{definition}

%Observe here that any $\sigma$-stable representation must be indecomposable, i.e., it cannot be written as a direct sum of (proper) subrepresentations. Indeed, suppose $V = V_1 \oplus V_2$, then $0 = \sigma(\underline{\dim} V) =  \sigma(\underline{\dim} V_1) + \sigma(\underline{\dim} V_2)$. Hence at least one of $\sigma(\underline{\dim} V_i) \geq 0$, and hence $V$ cannot be $\sigma$-stable. Also observe that if $V$ is a direct sum $V = V_1 \oplus V_2 \oplus \dots \oplus V_k$, then $V$ is $\sigma$-semistable (or $\sigma$-polystable) if and only if all the $V_i$ are. Thus, in order to understand whether a generic representation of dimension $\alpha$ is $\sigma$-semistable/polystable/stable, it is useful to understand how it decomposes as a direct sum of indecomposables, which is the topic of discussion in the next section. 

We now relate $\sigma$-stability notions to $\GL(\alpha)_\sigma$-stability notions. The following result is due to King. For a complete and self-contained proof, see the appendix in \cite{derksen2020maximum}. Note however that in \cite{derksen2020maximum}, the following result was stated for $\sigma$ indivisible, but it is clear from the proof that this is not necessary.

\begin{theorem} [King \cite{King}] \label{theo:King}
Let $Q$ be a quiver with no oriented cycles, $\alpha \in \Z_{>0}^{Q_0}$ a sincere dimension vector and $0 \neq \sigma \in \Z^{Q_0}$ such that $-\sigma \notin \Sigma(Q,\alpha)$. A representation $V \in \Rep(Q,\alpha)$ is $\sigma$-semistable (resp. $\sigma$-polystable, $\sigma$-stable) if and only if $V$ is $\GL(\alpha)_\sigma$-semistable (resp. $\GL(\alpha)_\sigma$-polystable, $\GL(\alpha)_\sigma$-stable).
\end{theorem} 
%King's original formulation is slightly different from the one above, but can be seen to be equivalent (details in Appendix~\ref{App.stability}). Now, we proceed to discuss these stability notions for the $m$-Kronecker quiver.

\begin{definition} 
We say a dimension vector $\alpha$ is $\sigma$-semistable/polystable/stable if a generic representation $V \in \Rep(Q,\alpha)$ is $\sigma$-semistable/polystable/stable. 
\end{definition}

%%%%%%%%%%%%%%%%%
\subsection{Canonical decomposition, Schur roots and Schofield's result}
%%%%%%%%%%%%%%%%%
In this subsection, we will recall some established results in the theory of quiver representations that we will need. First a series of definitions. For a representation $V$ of $Q$, we denote by $\End_Q(V)$ the endomorphism ring of $V$ (in the category of quiver representations).

\begin{definition} [Schur root] \label{def:schur-root}
Let $Q$ be a quiver and let $\alpha$ be a dimension vector. We say $\alpha$ is a Schur root if it satisfies any of the following equivalent conditions:
\begin{itemize}
\item $\End_Q(V) = \C$ for some $V \in \Rep(Q,\alpha)$.
\item $\End_Q(V) = \C$ for generic $V \in \Rep(Q,\alpha)$.
\item A generic $V \in \Rep(Q,\alpha)$ is indecomposable.
\end{itemize}
\end{definition}

We indicate briefly why the three conditions in the definition above are equivalent. The equivalence of the first two conditions follows from a standard argument in algebraic geometry using the fact that fiber dimension is upper semi-continuous. If $\End(V) = \C$, then $V$ must be indecomposable because the endomorphism ring is local, so this establishes that the second condition implies the third. The converse is non-trivial because the endomorphism rings of indecomposable representations are local and there are many other possiblities for local rings other than $\C$. We refer to \cite[Proposition~1]{Kac2} for a proof. 

The concept of canonical decomposition is due to Kac \cite{Kac, Kac2}.

\begin{definition}[Canonical decomposition]
Let $Q$ be a quiver and $\alpha$ a dimension vector. Then we write
$$
\alpha = \beta_1 \oplus \beta_2 \oplus \dots \oplus \beta_l
$$
and call it the canonical decomposition if a generic $V \in \Rep(Q,\alpha)$ splits as $V = V_1 \oplus V_2 \oplus \dots \oplus V_l$ with $V_i$ indecomposable of dimension $\beta_i$.
\end{definition}

Canonical decomposition will play an important role in our arguments in later sections. The following result is also due to Kac, but we first need a definition. For representations $V$ and $W$ of $Q$, we denote by $\Ext(V,W)$ the first extension group. We do not want to digress too much here about extension groups, so we refer the interested reader to \cite{DW-book}. For dimension vectors $\gamma $ and $\delta$ of $Q$, we denote by $\ext(\gamma,\delta)$ the generic value of $\dim(\Ext(V,W))$ for $(V,W) \in \Rep(Q,\gamma) \times \Rep(Q,\delta)$.

\begin{theorem} [Kac \cite{Kac2}] \label{thm:kac-ext}
Let $Q$ be a quiver and $\alpha$ a dimension vector. Then $\alpha = \beta_1 \oplus \beta_2 \oplus \dots \oplus \beta_l$ is the canonical decomposition if and only if both the following conditions hold:
\begin{itemize}
\item $\beta_i$ is a Schur root for all $i$;
\item $\ext(\beta_i,\beta_j) = 0$ for all $i \neq j$.
\end{itemize}
\end{theorem}

\begin{remark} \label{rmk:ext-gen-min}
It is again a consequence of the upper semi-continuity of fiber dimensions that $\ext(\gamma,\delta)$ is the minimum value of $\dim(\Ext(V,W))$ for $(V,W) \in \Rep(Q,\gamma) \times \Rep(Q,\delta)$.
\end{remark}

We now turn to recalling a result of Schofield. First, some definitions. For a quiver $Q$, we define the \emph{Euler form}, denoted $\left<-,-\right>$, on $\Z^{Q_0}$ as $\left<\alpha,\beta\right> = \sum_{x\in Q_0} \alpha(x) \beta(x) - \sum_{a \in Q_1} \alpha(ta) \beta(ha)$. The symmetrized version, often called Cartan form is defined as $(\alpha, \beta) = \left<\alpha,\beta\right> + \left<\beta,\alpha\right>$

% \CF{visu define symmetric euler form $(\cdot, \cdot)$ also}

\begin{definition} [Schofield weight and Schofield stability]
Let $Q$ be a quiver and $\alpha$ a dimension vector. The Schofield weight $\sigma$ is the weight defined by $\sigma(\beta) = \left<\alpha, \beta \right> - \left<\beta, \alpha \right>$. Further, we say that $\alpha$ is Schofield semistable/polystable/stable if $\alpha$ is $\sigma$-semistable/polystable/stable for the Schofield weight $\sigma$.
\end{definition}

Then reason for the above definition is clarified by the following theorem of Schofield which will be crucial for us in many ways. 

\begin{theorem} [Schofield] \label{thm:schofield}
Let $Q$ be a quiver and $\alpha$ a dimension vector. Then $\alpha$ is a Schur root if and only if $\alpha$ is Schofield stable.
\end{theorem}

%%%%%%%%%%%%%%%%%%%%%%%%%%%%%%%%%%%%%
\section{Relating iPCA to stability of star quivers} \label{sec:ipca-star}
%%%%%%%%%%%%%%%%%%%%%%%%%%%%%%%%%%%%%
In this section, we will see that the interpretation of iPCA as a Gaussian group model in Section~\ref{sec:ipca-gaussianmodel} is quiver theoretic in nature. The goal of this subsection is to relate (generic) existence and uniqueness of iPCA in a precise manner to Schofield stability for star quivers. This is quite straightforward for the complex version of iPCA, but we need some additional results to make the transfer to real iPCA. We stress that we relate maximum likelihood estimation for real iPCA directly to $\sigma$-stability (see Theorem~\ref{thm:realipca-to-sigma}), which is a property of complex representations of quivers rather than real representations of quivers! In the theory of quiver representations, complex representations are better understood than real representations, the underlying reason is an expected one -- the field of complex numbers is algebraically closed whereas the field of real numbers is not. 

\begin{remark}
Let $\alpha = (p,q_1,\dots,q_k)$ be a dimension vector for $\Delta[k]$, and let $n = \sum_i q_i$. Then, the Schofield weight of $\alpha$ is the weight $\sigma = (-n,p,p,\dots,p)$.
\end{remark}

We observe the following:
\begin{proposition} \label{prop:Malpha-to-sigma}
Let $K = \C$. Let $p,q_1,\dots,q_k \in \Z_{\geq 1}$ and let $n = \sum_i q_i$, let $\alpha = (p,q_1,\dots,q_k)$ be a dimension vector of $\Delta[k]$. Consider the complex model $\mathcal{M}_\alpha$. Then, for generic data $(B_1,\dots,B_k) \in \oplus_{i=1}^k \mat_{p,q_i}(\C)$:
\begin{itemize}
\item The likelihood function is bounded above  if and only if $\alpha$ is Schofield semistable;
\item MLE exists if and only if $\alpha$ is Schofield polystable;
\item MLE exists uniquely if and only if $\alpha$ is Schofield stable.
\end{itemize}
\end{proposition}

\begin{proof}
The Schofield weight $\sigma$ is $(-n,p,p,\dots,p)$, the same weight from Proposition~\ref{prop:ipca-to-gaussianmodel}. To show that $-\sigma \notin \Sigma(Q,\alpha)$, so that the hypotheses of Theorem~\ref{theo:King} are satisfied, observe that since there are no paths from $x$ to $y_i$, we deduce that $\SI(\Delta[k],\alpha)_{-d \sigma} = 0$ for all $d \in \Z_{\geq 1}$ from the determinantal description of semi-invariants of quivers (see \cite{DZ, SVd, DW-LR}). Now, the proposition follows from Proposition~\ref{prop:ipca-to-gaussianmodel} and Theorem~\ref{theo:King}.  
\end{proof}

In Proposition~\ref{prop:Malpha-to-sigma} above, we related the complex model $\mathcal{M}_{\alpha,\C}$ to $\sigma$-stability for the Schofield weight $\sigma$. But we are more interested in the real model $\mathcal{M}_{\alpha,\R}$ because the maximum likelihood estimation for this real model is precisely (real) iPCA, the main object of study in this paper. Hence, we prove the following result:

%So, we would like to establish an analog of the Proposition~\ref{prop:Malpha-to-sigma} in the case $K = \R$. However, the definition of $\alpha$ being $\sigma$-semistable/polystable/stable is only defined for complex representations of the quivers. Moreover, what we will now proceed to do is to relate the maximum likelihood estimation for the real model $\mathcal{M}_{\alpha,\R}$ (equivalently iPCA) to $\sigma$-stability!

\begin{theorem} \label{thm:realipca-to-sigma}
Let $p,q_1,\dots,q_k \in \Z_{\geq 1}$ and let $n =\sum_i q_i$ and let $\alpha = (p,q_1,\dots,q_k)$ be a dimension vector of $\Delta[k]$. Then, for generic data $B = (B_1,\dots,B_k) \in \oplus_{i=1}^m \mat_{p,q_i}(\R)$:
\begin{itemize}
\item Eq~\ref{eq:ipca-def} is bounded above if and only if $\alpha$ is Schofield semistable;
\item iPCA exists if and only if $\alpha$ is Schofield polystable;
\item iPCA exists uniquely if and only if $\alpha$ is Schofield stable.
\end{itemize}
\end{theorem}

Theorem~\ref{thm:realipca-to-sigma} is not as straightforward as Proposition~\ref{prop:Malpha-to-sigma} and we need additional insights to make the `transfer' from $K = \C$ to $K= \R$. This requires some technical results about varieties and algebraic groups defined over $\R$. Since, we do not need these technical details in later sections, we move them and the proof of Theorem~\ref{thm:realipca-to-sigma} to Appendix~\ref{app: algebraic groups over R} for the sake of readability.

%%%%%%%%%%%%%%%%%%%%%%%%%%%%%%%%%%%%%%%%%
\section{Schofield stability for star quivers}\label{sec:scho-star}
%%%%%%%%%%%%%%%%%%%%%%%%%%%%%%%%%%%%%%%%%
In this section, we will prove some of our main results on existence and uniqueness of iPCA by studying Schofield stability. To do so, we will need to introduce additional techniques from quiver representations, which we will discuss when needed. We set $K = \C$ as the ground field for this entire section.

\subsection{Equivalence of Schofield semistability and Schofield polystability for star quivers}
First, we discuss a result that is specific to the quiver $\Delta[k]$, which may not be true for more general quivers.

\begin{proposition} \label{prop:sigma-ss-ps}
Let $p,q_1,\dots,q_k \in \Z_{>0}$, and let $\alpha = (p,q_1,\dots,q_k)$ be a dimension vector of the star quiver $\Delta[k]$. Then
$$
\alpha \text{ is Schofield semistable} \Longleftrightarrow \alpha \text{ is Schofield polystable}.
$$
\end{proposition}

To prove the proposition above, we need the following lemma.

\begin{lemma} \label{lem:sigma-ps-can}
Let $p,q_1,\dots,q_k \in \Z_{>0}$, and let $\alpha = (p,q_1,\dots,q_k)$ be a dimension vector of the star quiver $\Delta[k]$. Let $\sigma$ be the Schofield weight of $\alpha$. Let $\alpha = \beta_1 \oplus \beta_2 \oplus \dots \oplus \beta_l$ be the canonical decomposition of $\alpha$. Then
$$
\alpha \text{ is $\sigma$-polystable} \Longleftrightarrow \sigma \cdot \beta_i = 0\ \forall i.
$$
\end{lemma}

\begin{proof}
Suppose $\alpha$ is $\sigma$-polystable. Then, clearly, we must have $\sigma \cdot \beta_i \leq 0$ for all $i$ since $\alpha$ is $\sigma$-semistable. But then, we also have $\sigma \cdot (\sum_i \beta_i) = \sigma \cdot \alpha = 0$. Thus, we must have $\sigma \cdot \beta_i = 0$ for all $i$.

Conversely, suppose $\sigma \cdot \beta_i = 0$ for all $i$. Fix $j$. Since $\sigma(x) < 0$ and $\sigma(y_i) > 0$ for all $i$ and $\beta_j \neq 0$, the support of $\beta_j$ must be of the form $S = \{x,y_{j_1},\dots,y_{ij_k}\}$ for some $\emptyset \neq J = \{j_1,\dots,j_k\} \subseteq \{1,2,\dots,k\}$. Consider $\Delta[J]$ the (full) subquiver of $\Delta[k]$ whose vertices are $\{x\} \cup \{y_j\ |\ j \in J\}$. Then, we can interpret $\beta_j$ as a dimension vector of $\Delta[J]$. Since $\beta_j$ is a Schur root (of $\Delta[k]$ and hence of $\Delta[J]$), it is $\mu$-stable for its Schofield weight $\mu$ by Theorem~\ref{thm:schofield}. But one easily checks that $\sigma \cdot \beta_j = 0$ means that $\mu$ is (up to a positive multiple) equal to $\sigma$ restricted to $\Delta[J]$, so $\beta_j$ is $\sigma$-stable. Since $\beta_j$ is $\sigma$-stable for all $j$, we conclude that $\alpha$ is $\sigma$-polystable.
\end{proof}

Now, we can prove Proposition~\ref{prop:sigma-ss-ps}.

\begin{proof} [Proof of Proposition~\ref{prop:sigma-ss-ps}]
Clearly if $\alpha$ is $\sigma$-polystable, then $\alpha$ is $\sigma$-semistable. Conversely, suppose $\alpha$ is not $\sigma$-polystable. Let $\alpha = \beta_1 \oplus \beta_2 \oplus \dots \oplus \beta_l$ be the canonical decomposition. By Lemma~\ref{lem:sigma-ps-can}, we see that $\sigma \cdot \beta_i \neq 0$ for some $i$. Now either $\sigma \cdot \beta_i > 0$ or $\sigma \cdot (\sum_{ j \neq i} \beta_j) > 0$. In either case, this means that a generic $V \in \Rep(Q,\alpha)$ has a direct summand $W$(and hence subrepresentation) such that $\sigma \cdot \dim(W) > 0$, so $V$ is not $\sigma$-semistable.
\end{proof}
The above results raise an interesting question in quiver representations:

\begin{problem}
Does there exist a quiver $Q$ and a dimension vector $\alpha$ such that $\alpha$ is Schofield semistable, but not Schofield polystable? If so, construct an example.
\end{problem}

\subsection{Opposite quivers}
There are some operations we can perform on quivers while preserving information about their generic stability/semistability properties in order to extend the number of cases we can handle. One such operation is as follows. Let $Q = (Q_0,Q_1)$ be a quiver, and let $Q^{\opp}$ denote the \emph{opposite quiver}, where $Q^{\opp}_0 = Q_0$ and $Q^{\opp}_1 = \{a^* \ |\ a \in Q_1\}$ where $ta^* = ha$ and $ha^* = ta$. In other words, $Q^{\opp}$ is obtained from $Q$ by simply reversing the orientations of all arrows. Given a representation $V$ of $Q$, we define a representation $V^{\opp}$ of $Q^{\opp}$ as follows. For each $x \in Q_0 = Q^{\opp}_0$, set $V^{\opp}_x = V_x^*$ (the dual of $V_x$) and for each $a^* \in Q^{\opp}_1$, let $V(a^*): V(ta^*) = V(ha)^* \rightarrow V(ta)^* = V(ha^*)$ be the dual of the linear transformation $V(a): V(ta) \rightarrow V(ha)$. It is easy to check that the category of representations of $Q^{\opp}$ is the equivalent to the dual (or opposite) category to the category of representations of $Q$ and this anti-equivalence is given by the contravariant functor $V \mapsto V^{\opp}$.

\begin{lemma} \label{lem:ref-sigma-minus}
Let $V$ be a representation of a quiver $Q$. Let $\sigma \in \Z^{Q_0}$ be a weight. Then $V$ is $\sigma$-semistable/polystable/stable if and only if $V^{\opp}$ is $(-\sigma)$-semistable/polystable/stable. 
\end{lemma}

\begin{proof}
The functor $V \mapsto V^{\opp}$ is contravariant and an anti-equivalence, so it turns subrepresentations into quotients and vice versa (while preserving the dimension vectors). Using this and Definition~\ref{crit-king}, the required result follows easily.
\end{proof}

For a representation $V$ of $Q$, identify $V_x$ with $\C^{\dim(V_x)}$ for all $x \in Q_0$ by choosing bases. By giving the standard inner product to each $\C^{\dim(V_x)}$, we can identify its dual again with $\C^{\dim(V_x)}$. With these identifications, the matrix representing $V(a^*)$ is the conjugate transpose of the matrix representing $V(a)$.
To summarize, we have a bijective map
\begin{align*}
\eta: \Rep(Q,\alpha) & \rightarrow  \Rep(Q^{\opp}, \alpha) \\
(V(a))_{a \in Q_1}  & \mapsto  (V(a)^{\dag})_{a^* \in Q^{\opp}_1}
\end{align*}

\begin{lemma} \label{lem:stability-opposite-negative}
Let $Q$ be a quiver and $\alpha \in \Z^{Q_0}$ a dimension vector and $\sigma \in \Z^{Q_0}$ a weight. Then $\alpha$ is $\sigma$-semistable/polystable/stable for $Q$ if and only if $\alpha$ is $(-\sigma)$-semistable/polystable/stable for the quiver $Q^{\opp}$.
\end{lemma}

\begin{proof}
This follows from the fact that the map $\eta$ defined above is bijective and converts $\sigma$-semistability/polystability/stability to $(-\sigma)$-semistability/polystability/stability by Lemma~\ref{lem:ref-sigma-minus}.
\end{proof}

\begin{remark} \label{rmk:schofieldweight-opposite}
Note that given a dimension vector $\alpha$ of a quiver $Q$, the Schofield weight of $\alpha$ for $Q^{\opp}$ is the negative of the Schofield weight of $\alpha$ for $Q$.
\end{remark}

Recall the star quiver $\Delta[k]$. Its opposite $\Delta[k]^{\opp}$ is in the figure below:
\begin{equation}\begin{tikzcd}
& x & &  \\
y_1 \arrow[ur, <-, "a_1^*"] &  y_2\arrow[u, <-,"a_2^*"] &  \dots & y_n \arrow[ull, <-,"a_k^*" above]
%\& \mu(X)_{2} \& \&
\end{tikzcd}\label{eq:star-quiver-opp}\end{equation}

The transition from a quiver to its opposite is not very powerful in itself and seems like an exercise in linear algebra about dualizing vector spaces and linear maps. However, there are other more interesting ways to go from $\Delta[k]$ to $\Delta[k]^{\opp}$ called reflection functors. The interplay of the above duality with reflection functors helps us completely resolve generic existence and uniqueness of iPCA in the case where the legs are equal, i.e., $q_1 = q_2 = \dots = q_k$. Now, we turn to discussing reflection functors.

\subsection{Reflection functors}
Reflection functors are widely used in the theory of quiver representations and were discovered by Bernstein, Gelfand and Ponomarev. Despite their simplicity, reflection functors seem to be surprisingly powerful in numerous contexts. We recall only what we need from reflection functors and refer the interested reader to \cite{DW-book} for a more comprehensive treatment.

Given a quiver $Q = (Q_0,Q_1)$ and a vertex $x \in Q_0$, we define \emph{the quiver reflected at $x$}, denoted $r_x(Q)$, whose vertex set is $Q_0$ and set of arrows is 
$$
(r_x(Q))_1 = \{a \in Q_1\ |\ ha \neq x, ta \neq x\} \cup \{a^*\ |\ a \in Q_1, ha = x \text{ or } ta = x\},
$$
and $ha^* = ta$ and $ta^* = ha$. In other words, one simply reverses all arrows which ``touch" the vertex $x$. 

\subsubsection{Reflection at a sink}Now, suppose $x$ is a sink of the quiver $Q$ (i.e., there are no arrows such that $ta = x$, such as the vertex we have labeled $x$ in the star quiver $\Delta[k]$) and let $a_1,\dots,a_l$ be the arrows in $Q$ such that $ha_i = x$, and let $y_i = ta_i$. Suppose $V$ is a representation of $Q$. Then we have a map 
$$
\phi(x) = (V(a_1),V(a_2), \dots, V(a_l))^\top: \bigoplus_{i=1}^l V(y_i) \rightarrow V(x).
$$

Thus we have the inclusion $\iota_x: {\rm ker}(\phi_x) \hookrightarrow  \bigoplus_{i=1}^l V(y_i)$. For $j \in \{1,2,\dots,l\}$, let $\pi_j:  \bigoplus_{i=1}^l V(y_i) \rightarrow V_{y_j}$ denote the projection to the $j^{th}$ summand. Then $\pi_j \circ \iota_x : {\rm ker}(\phi_x) \rightarrow V(y_j)$.

We now define a representation $W = R_x^+(V)$ of $r_x(Q)$ as follows: For $y \neq x$, set $W(y) = V(y)$ and set $W(x) = {\rm ker}(\phi_x)$. For $b \in Q_1 \setminus \{a_1,\dots,a_k\}$, set $W(b) = V(b)$ and for $j \in \{1,2,\dots,l\}$, we define $W(a_j^*) = \pi_j \circ \iota_x$.

\subsubsection{Reflection at a source} If $x$ is a source of the quiver $Q$ (i.e., there are no arrows such that $ha = x$) and let $a_1,\dots,a_l$ be the arrows in $Q$ such that $ta_i = x$, and let $y_i = ha_i$. Suppose $V$ is a representation of $Q$. Then we have a map 
$$
\psi(x) = (V(a_1),V(a_2), \dots, V(a_l)): V(x) \rightarrow \bigoplus_{i=1}^l V(y_i)
$$

Thus, we have a surjection $\eta_x :  \bigoplus_{i=1}^l V(y_i) \rightarrow {\rm coker}(\psi_x)$. For $j \in \{1,2,\dots,l\}$, let $\zeta_j : V(y_j) \hookrightarrow \bigoplus_{i=1}^l V(y_i)$ denote the inclusion of the $j^{th}$ summand. Then $\eta_x \circ \zeta_j : V(y_j) \rightarrow {\rm coker}(\psi_x)$.

We now define a representation $W = R_x^-(V)$ of $r_x(Q)$ as follows: For $y \neq x$, set $W(y) = V(y)$ and set $W(x) = {\rm coker}(\psi_x)$. For $b \in Q_1 \setminus \{a_1,\dots,a_k\}$, set $W(b) = V(b)$ and for $j \in \{1,2,\dots,l\}$, we define $W(a_j^*) = \eta_x \circ \zeta_j$.

\subsubsection{Reflection functors and canonical decomposition}
It can be checked that $R_x^+$ and $R_x^-$ (when defined) are functors from the category of representations of $Q$ to the category of representations of $r_x(Q)$, the quiver $Q$ reflected at $x$. We will recall the important results.

For $x \in Q_0$, let $S_x$ denote the representation that is defined as $(S_x)_x = K$ and $(S_x)_y = 0$ for all $x \neq y \in Q_0$. We need not describe the maps between vertices, because all of them must be $0$.

\begin{lemma}
Let $Q$ be a quiver and let $x \in Q_0$ be a sink. Let $V$ be an indecomposable representation of $Q$ which is not isomorphic to $S_x$. Then $R_x^+(Q)$ is an indecomposable representation of $r_x(Q)$ and moreover $R_x^-(R_x^+(V)) = V$.
\end{lemma}

Let $Q$ be a quiver and $x \in Q_0$ be a sink. Let $\mathcal{C}_{Q,x}$ denote the subcategory of representation of the quiver $Q$ which do not have a direct summand isomorphic to $S_x$. Similarly, let $\mathcal{C}_{r_x(Q),x}$ denote the subcategory of representations of $r_x(Q)$ which do not have a direct summand isomorphic to $S_x$. Then, we have an equivalence of categories

$$ \begin{tikzcd}[ampersand replacement = \&] \mathcal{C}_{Q,x}  \arrow[bend left]{r}{R_x^+}   \& \arrow[bend left]{l}[pos=0.4]{R_x^-} \mathcal{C}_{r_x(Q),x} 
\end{tikzcd}$$

given by $R_x^+$ and $R_x^-$. Moreover, for $V \in \mathcal{C}_{Q,x}$ of dimension $\alpha$, we observe that the dimension of $R_x^+(V)$ is $r_x(\alpha)$. For a vertex $x$ in a quiver $Q$, let $\epsilon_x$ denote the dimension vector which is $1$ at $x$ and $0$ elsewhere. 
%\VM{Explain somewhere at the beginning that modulo these $S_x$'s, reflection functors transfer information perfectly to the reflected quiver}

%\VM{Need to state that $\alpha$ is Schur root $\Leftrightarrow$ generic endomorphism ring of $\alpha$-dimensional representation is trivial $\Leftrightarrow$ there exists some $\alpha$-dimensional representation with a trivial endomorphism ring. Probably can find it somewhere in Kac.}

\begin{lemma} \label{lem:reflection-schur}
Let $Q$ be a quiver and let $x \in Q_0$ be a sink. Let $\alpha \neq \epsilon_x$ be a Schur root for $Q$. Then $r_x(\alpha)$ is a Schur root for $r_x(Q)$.
\end{lemma}

\begin{proof}
Let $V$ be an indecomposable representation of dimension $\alpha$ such that ${\rm End}_{Q}(V) = \C$. Such a $V$ exists because $\alpha$ is a Schur root for $Q$, see Definition~\ref{def:schur-root}. Since $V$ is indecomposable, it does not have a direct summand isomorphic to $S_x$, so $V \in \mathcal{C}_{Q,x}$. Hence $R_x^+(V)$ is a representation of $r_x(Q)$ of dimension $r_x(\alpha)$ with a whose endomorphism ring is $\C$, so we conclude that $r_x(\alpha)$ is a Schur root for $r_x(Q)$.
\end{proof}

%\VM{Need Kac's characterization of canonical decomposition in terms of Schur roots and vanishing of extensions.}

\begin{proposition} \label{prop:can-dec-reflect}
Let $Q$ be a quiver and let $x \in Q_0$ be a sink. Let $\alpha$ be a dimension vector of $Q$. Suppose its canonical decomposition is $\alpha = \beta_1 \oplus \beta_2 \oplus \dots \oplus \beta_l$ (as a dimension vector of $Q$) such that $\beta_i \neq \epsilon_x$ for all $i$. Then
$$
r_x(\alpha) = r_x(\beta_1) \oplus r_x(\beta_2) \oplus \dots \oplus r_x(\beta_l)
$$
is the canonical decomposition of $r_x(\alpha)$ as a dimension vector of $r_x(Q)$.
\end{proposition}

\begin{proof}
First, we observe that $r_x(\beta_i)$ are all Schur roots for $r_x(Q)$ by Lemma~\ref{lem:reflection-schur}. It suffices to show that $\ext_{r_x(Q)}(r_x(\beta_i),r_x(\beta_j)) = 0$ for all $i\neq j$ because then we get the required conclusion by Theorem~\ref{thm:kac-ext}.  Fix $i$ and $j$. Since $\alpha = \beta_1 \oplus \beta_2 \oplus \dots \oplus \beta_l$ is the canonical decomposition for $Q$, we get that $\ext_Q(\beta_i,\beta_j) = 0$. In particular, this means that for generic representations $V_i$ and $V_j$ of dimension $\beta_i$ and $\beta_j$ such that $\Ext_Q(V_i,V_j) = 0$. Since $\beta_i$ and $\beta_j$ are Schur roots, we can assume $V_i$ and $V_j$ are indecomposable and in particular this implies that $V_i,V_j \in \mathcal{C}_{Q,x}$. Thus, since $R_x^+$ gives an equivalence of categories from $C_{Q,x}$ to $C_{r_x(Q),x}$, we get that $\Ext_{r_x(Q)}(r_x(V_i), r_x(V_j)) = 0$, so we conclude that $\ext_{r_x(Q)}(r_x(\beta_i), r_x(\beta_j)) = 0$ (see Remark~\ref{rmk:ext-gen-min}). 
%\VM{Check that Exts are transferred through reflection functors. I think it's automatic, but check anyway to be completely sure.}
\end{proof}

\begin{remark}
Analogous statements can be made for $R_x^-$ when $x$ is a source of the quiver $Q$. We will not write them out explicitly, but we will make use of them whenever needed.
\end{remark}

\subsubsection{Reflection functors for star quivers}

Let $\alpha = (p,q_1,\dots,q_k)$ be a dimension vector for $\Delta[k]$. Observe that $x$ is a sink. Now, observe that $r_x(\alpha) = (n-p, q_1,q_2,\dots,q_k)$. 

\begin{lemma} \label{lem:p<n-noSx}
Let $\alpha = (p,q_1,\dots,q_k)$ be a dimension vector for $\Delta[k]$ and suppose $p < n$. Then, a generic representation $V \in \Rep(\Delta[k],\alpha)$ does not contain $S_x$ as a direct summand.
\end{lemma}

\begin{proof}
This is straightforward and left to the reader.
\end{proof}

\begin{lemma} \label{lem:schofieldweight-identify}
Let $\alpha = (p,q_1,\dots,q_k)$ be a dimension vector of $\Delta[k]$. Suppose $\sigma \in \Z^{Q_0}$ is such that 
\begin{itemize}
\item $\sigma \cdot \alpha = 0$;
\item $\sigma(x) < 0$;
\item $\sigma(y_1) = \sigma(y_2) = \dots = \sigma(y_k) > 0$.
\end{itemize}
Then up to a positive multiple $\sigma$ is the Schofield weight of $\alpha$.
\end{lemma}

\begin{proof}
This is straightforward and left to the reader.
\end{proof}

\begin{corollary} \label{cor:schofieldstable-reflection}
Let $\alpha = (p,q_1,\dots,q_k)$ be a dimension vector for $\Delta[k]$ such that $p < n$. Then $\alpha$ is Schofield semistable/polystable/stable for $\Delta[k]$ if and only if $r_x(\alpha)$ is Schofield semistable/polystable/stable for $\Delta[k]^{\opp}$.
\end{corollary}

%Let $\sigma = (n,-p,-p,\dots,-p)$ and let $\widehat{\sigma} = (-n, n-p,n-p,\dots,n-p)$. 
% \CF{this sign convention disagrees with 2.7, right?} \VM{Yeah, I suppose I missed fixing it, fixed it now. Also fixed in Lemma above}

\begin{proof}
We make a few observations before we get into the proof. First, the Schofield weight for $\alpha$ as a representation of $\Delta[k]$ is $\sigma = (-n,p,p,\dots,p)$ and the Schofield weight for $r_x(\alpha) = (n-p, q_1,\dots,q_k)$ as a representation of $\Delta[k]^{\opp}$ is $\widehat{\sigma} = (n,p-n,p-n,\dots,p-n)$. Second, if $\alpha = \beta_1 \oplus \dots \oplus \beta_l$ is the canonical decomposition of $\alpha$ for $\Delta[k]$, then  $r_x(\alpha) = r_x(\beta_1) \oplus \dots \oplus r_x(\beta_l)$ is the canonical decomposition of $r_x(\alpha)$ for $\Delta[k]^{\rm opp}$, which follows from Lemma~\ref{lem:p<n-noSx} and Proposition~\ref{prop:can-dec-reflect}. Third, Schofield semistability is the same as Schofield polystability for $\Delta[k]$ by Proposition~\ref{prop:sigma-ss-ps}. Analogously, one can establish that Schofield semistability is the same as Schofield polystability for $\Delta[k]^{\opp}$; we leave the details to the reader. So, we only need to argue the cases of polystability and stability.

First, we tackle the case of stability. Suppose $\alpha$ is Schofield stable, i.e., $\sigma$-stable for $\Delta[k]$. Then since $\alpha$ is a Schur root by Theorem~\ref{thm:schofield}, and $\alpha \neq \epsilon_x$, we deduce that $r_x(\alpha)$ is a Schur root for $\Delta[k]^{\opp}$ by Lemma~\ref{lem:reflection-schur}. Thus $r_x(\alpha)$ is Schofield stable for $\Delta[k]^{\opp}$ by Theorem~\ref{thm:schofield}. Conversely, suppose $\alpha$ is not $\sigma$-stable for $\Delta[k]$, then it is not Schur, so its canonical decomposition contains at least two factors, none of which are $\epsilon_x$ by Lemma~\ref{lem:p<n-noSx}. But then by the argument above, the canonical decomposition for $r_x(\alpha)$ for $\Delta[k]^{\opp}$ also contains at least two factors, so $r_x(\alpha)$ is not a Schur root and hence not Schofield stable.

Next, we tackle the case of polystability. Suppose $\alpha$ is Schofield polystable, i.e., $\sigma$-polystable for $\Delta[k]$. Then the canonical decomposition $\alpha = \beta_1 \oplus \beta_2 \oplus \dots \oplus \beta_l$ is such that each $\beta_i$ is $\sigma$-stable and $\beta_i \neq \epsilon_x$. In particular, observe that this means $\sigma \cdot \beta_i = 0$ and this implies that the Schofield weight of $\beta_i$ for $Q$ is $\sigma$ up to a positive multiple by Lemma~\ref{lem:schofieldweight-identify}. One also verifies easily that $\sigma \cdot \beta_i = 0 \implies \widehat{\sigma} \cdot r_x(\beta_i) = 0$, which means that the Schofield weight of $r_x(\beta_i)$ for $r_x(Q)$ is $\widehat{\sigma}$ up to a positive multiple by the analog of Lemma~\ref{lem:schofieldweight-identify} for $\Delta[k]^{\opp}$. Since $r_x(\beta_i)$ is a Schur root, it is Schofield stable, i.e., $\widehat{\sigma}$-stable. Hence $r_x(\alpha)$ is $\widehat{\sigma}$-polystable, i.e., Schofield polystable.
\end{proof}

\begin{corollary} \label{cor:r_x-stability-preserve}
Let $\alpha = (p,q_1,\dots,q_k)$ be a dimension vector for $\Delta[k]$ such that $p < n$. Then $\alpha$ is Schofield semistable/polystable/stable for $\Delta[k]$ if and only if $r_x(\alpha) = (n-p,q_1,\dots,q_k)$ is Schofield semistable/polystable/stable for $\Delta[k]$.
\end{corollary}

\begin{proof}
This follows from Corollary~\ref{cor:schofieldstable-reflection}, Lemma~\ref{lem:stability-opposite-negative} and Remark~\ref{rmk:schofieldweight-opposite}.
\end{proof}

In the above, we focused on reflection at the sink $x$ in the quiver $\Delta[k]$. We can also look at reflections at the sources $y_i$, i.e., $R_{y_i}^-$. Applying $R_{y_i}^-$ reverses only the orientation of one arrow. But then we observe that $R_{y_i}^-$ and $R_{y_j}^-$ commute and we define $R_y^- = R_{y_1}^- \circ R_{y_2}^- \circ \dots \circ R_{y_k}^-$. Now, $R_y^-$ is also a functor that takes representations of $\Delta[k]$ to representations of $\Delta[k]^{\opp}$. Using similar arguments as above, one can prove the following (whose proof we leave to the reader):

\begin{proposition} \label{prop:r_y-stability-preserve}
Let $\alpha = (p,q_1,\dots,q_k)$ be a dimension vector for $\Delta[k]$ such that $q_i < p$ for all $i$. Then $\alpha$ is Schofield semistable/polystable/stable for $\Delta[k]$ if and only if $(p, p - q_1,\dots,p - q_k)$ is Schofield semistable/polystable/stable for $\Delta[k]$.
\end{proposition}

\subsection{Proofs of Theorem~\ref{thm:main} and Theorem~\ref{thm:same-dim}}

In order to prove Theorem~\ref{thm:main}, we need to recall a result of Kac. We say a dimension vector $\alpha$ of a quiver $Q$ is connected if the directed graph induced by its support is connected.

\begin{theorem} [\cite{Kac}]\label{thm:kac-fundamentalchamber}
Let $Q = (Q_0,Q_1)$ be a quiver. Let $\alpha$ be a connected dimension vector such that $(\alpha,\epsilon_x) \leq 0$ for all $x \in Q_0$.
\begin{itemize}
\item If $(\alpha,\epsilon_x) < 0$ for some $x \in Q_0$, then $\alpha$ is a Schur root.
\item If $(\alpha,\epsilon_x) = 0$ for all $x \in Q_0$, then we can write $\alpha = c \beta$ uniquely for $c \in \Z_{\geq 1}$ and $\beta$ indivisible. Further;
\begin{itemize}
\item If $c = 1$, then $\alpha$ is a Schur root
\item If $c > 1$, then $\alpha = \beta^{\oplus c}$ is the canonical decomposition of $\alpha$.
\end{itemize}
\end{itemize}
\end{theorem}

Now we can prove Theorem~\ref{thm:main}.

\begin{proof}[Proof of Theorem~\ref{thm:main}]
Consider the dimension vector $\alpha = (p,q,q,\dots,q)$ of $\Delta[k]$. Observe that $(\alpha, \epsilon_x) = 2p - \sum_i q_i \leq 0$ and $(\alpha, \epsilon_{y_i}) = 2q_i - p \leq 0$. It is easy to see that the only time we have $(\alpha,\epsilon_x)$ = $(\alpha, \epsilon_{y_i})= 0$ for all $i$ is if $(p,q_1,\dots,q_k) = c(2,1,1,1,1)$. Unless $\alpha = (2c,c,c,c,c)$, we have that $\alpha$ is a Schur root by part ($1$) of Theorem~\ref{thm:kac-fundamentalchamber}, hence Schofield stable by Theorem~\ref{thm:schofield}, and so iPCA generically exists uniquely by Theorem~\ref{thm:realipca-to-sigma}.

In the case $\alpha = (2c,c,c,c,c)$, by part ($2$) of Theorem~\ref{thm:kac-fundamentalchamber}, we deduce that if $c = 1$, then $\alpha$ is a Schur root and hence iPCA generically exists uniquely by Theorem~\ref{thm:schofield} and Theorem~\ref{thm:realipca-to-sigma} (as argued above). If $c > 1$, then since the canonical decomposition is $(2c,c,c,c,c) = (2,1,1,1,1)^{\oplus c}$ by part ($2$) of Theorem~\ref{thm:kac-fundamentalchamber} and $(2,1,1,1,1)$ is Schur, so Schofield stable. This means that $\alpha = (2c,c,c,c,c)$ is Schofield polystable (but not Schofield stable), so we conclude by Theorem~\ref{thm:realipca-to-sigma} that iPCA generically exists, but we do not have generic uniqueness.
\end{proof}

Now, we turn to proving Theorem~\ref{thm:same-dim}. For $p,q \in \Z_{>1}$, define $\alpha(p,q) = (p,q,\dots,q) \in \Z^{k+1}$ to be interpreted as a dimension vector of $\Delta[k]$. Then, by Corollary~\ref{cor:r_x-stability-preserve}, if $p < kq$, then $\alpha(p,q)$ is Schofield semistable/polystable/stable for $\Delta[k]$ if and only if $\alpha(kq-p,q)$ is Schofield semistable/polystable/stable for $\Delta[k]$. Similarly, by Proposition~\ref{prop:r_y-stability-preserve}, if $q < p$, then $\alpha(p,q)$ is Schofield semistable/polystable/stable for $\Delta[k]$ if and only if $\alpha(p,p-q)$ is Schofield semistable/polystable/stable for $\Delta[k]$. 

In view of the above paragraph we want to define an equivalence relation $\sim_k$ on $\Z_{>1}^2$. If $p < kq$, then $(p,q) \sim_k (kq-p, q)$ and if $q < p$, then $(p,q) \sim_k (p,p-q)$ and taking the transitive closure defines $\sim_k$.

\begin{lemma}
If $(p,q) \sim_k (r,s)$, then $\alpha(p,q)$ is Schofield semistable/polystable/stable if and only if $\alpha(r,s)$ is Schofield semistable/polystable/stable.
\end{lemma}

\begin{proof}
This follows from Proposition~\ref{prop:r_y-stability-preserve} and Corollary~\ref{cor:r_x-stability-preserve}.
\end{proof}

\begin{lemma}
Consider the quantity $\Gamma (p,q) = p^2 + kq^2 - kpq$. If $(p,q) \sim_k (r,s)$, then $\Gamma(p,q) = \Gamma(r,s)$.
\end{lemma}

\begin{proof}
It is enough to check that $\Gamma(p,q) = \Gamma(kq-p,q)$ when $p < kq$ and that $\Gamma(p,q) = \Gamma(p,p-q)$ when $q < p$. Both are straightforward computations and left to the reader.
\end{proof}

\begin{definition}
Let $Q$ be a quiver and $\alpha$ a Schur root. We say $\alpha$ is a real (resp. isotropic imaginary, non-isotropic imaginary) if $\left<\alpha,\alpha\right> = 1$ (resp. $0$, $<0$). 
\end{definition}

\begin{theorem} \label{thm:Schofield-can}
Let $Q$ be a quiver and $\alpha$ a dimension vector. Suppose $\alpha = \beta_1 \oplus \beta_2 \oplus \dots \oplus \beta_l$ is the canonical decomposition of $\alpha$, then the canonical decomposition of $n \alpha$ is
$$
n\alpha = (n\beta_1) \oplus (n\beta_2) \oplus \dots \oplus (n\beta_l),
$$
where 
$$(n\beta) = \begin{cases}
\beta^{\oplus n} & \text{if $\beta$ is real or isotropic};\\
n\beta & \text{if $\beta$ is non-isotropic}.
\end{cases}
$$

\end{theorem}

\begin{proposition} \label{prop:castling}
Let $(p,q) \in \Z_{>1}^2$. Then:
\begin{enumerate}
\item If $q \leq p/2$ and $p \leq kq/2$, then $\alpha(p,q)$ is Schofield stable unless $k = 4$, $p = 2q$ and $q > 1$.
\item If $q > p$ or $p > kq$, then $\alpha(p,q)$ is not Schofield semistable.
\item If $q = p$ or $p = kq$, then $\alpha(p,q)$ is Schofield polystable. Further it is Schofield stable if and only if $(p,q) = (1,1)$.
\item In all other cases, we have either $p/2 < q < p$ or $kq/2 < p < kq$ and $(p,q)$ is $\sim_k$-equivalent to some $(p',q')$ which falls in the above three cases.
\end{enumerate}

\end{proposition}

\begin{proof}
The first part is Theorem~\ref{thm:main}. For the second part, observe that if $q > p$, then $\epsilon_{y_1} = (0,1,0,\dots,0)$ appears in the canonical decomposition of $\alpha$ and $\sigma \cdot \epsilon_{y_1} \neq 0$ where $\sigma = (kq, -p,-p,\dots,-p)$ is the Schofield weight of $\alpha(p,q)$. Thus, by Lemma~\ref{lem:sigma-ps-can}, $\alpha(p,q)$ is not Schofield polystable, and hence not Schofield semistable by Proposition~\ref{prop:sigma-ss-ps}. Similarly if $p > kq$, then $\epsilon_x = (1,0,\dots,0)$ appears in the canonical decomposition for $\alpha(p,q)$ and again $\sigma \cdot \epsilon_x \neq 0$. Thus, by a similar argument, we get that $\alpha(p,q)$ is not Schofield semistable.

%\VM{Need to recall somewhere real, isotropic and imaginary Schur roots and Schofield's result on canonical decomposition}.

For the third part, if $p = q$, then $\alpha(p,q) = (1,1,\dots,1)^{\oplus p}$ is the canonical decomposition of $\alpha(p,q)$. This follows from the fact that $(1,1,\dots,1)$ is a real Schur root. It is straightforward to verify that $(1,1,\dots,1)$ is a real Schur root and we leave that to the reader. Hence if $p = q = 1$, $\alpha(p,q)$ is Schofield stable and if $p = q > 1$ then $\alpha(p,q)$ is Schofield polystable, but not Schofield stable. Now, the case $p = kq$. If $q = 1$, then we have the canonical decomposition:
$$
\alpha(k,1) = (1,1,0\dots,0) \oplus (1,0,1,0,\dots,0) \oplus \dots \oplus (1,0,\dots,0,1).
$$
which can be argued directly. (for e.g., see the argument in the proof of Proposition~7.2 in \cite{derksen2020maximum}). Then, by Lemma~\ref{lem:sigma-ps-can}, we get that $\alpha(k,1)$ is $\sigma$-polystable and $\sigma$-stable precisely when $k = 1$, i.e., $k = 1, (p,q) = (1,1)$. Now, if $p = kq$ and $q > 1$, then by Theorem~\ref{thm:Schofield-can}, we deduce the canonical decomposition
$$
\alpha(kq,q) = (1,1,0\dots,0)^{\oplus q} \oplus (1,0,1,0,\dots,0)^{\oplus q} \oplus \dots \oplus (1,0,\dots,0,1)^{\oplus q}.
$$

Thus $\alpha(kq,q)$ is Schofield polystable, but not Schofield stable by Lemma~\ref{lem:sigma-ps-can}.

Now, for the last and fourth part. It is clear that in all other cases, we $p/2 < q < p$ or $kq/2 < p < kq$. For $(r,s) \in \Z_{>1}^2$, let $\delta(r,s) = krs$ and call $(r,s)$ minimal if it minimizes $\delta$ in its equivalence class.  In the former case (i.e., $p/2 < q < p$), since $p > q$,  we have $(p,q) \sim_k (p,p-q)$ and $\delta(p,q) = kpq >  kp(p-q) = \delta(p,p-q)$ because $q > p/2$. In the latter case, we have $(p,q) \sim_k (kq-p,q)$ and $\delta(p,q) = kpq > kp(kq-p) = \delta(kq-p,q)$ because $p > kq/2$. Thus we have shown that $p/2 < q < p$ or $kq/2 < p < kq$ means that $(p,q)$ is not minimal. Now, if we let $(p',q')$ be a minimal element in its equivalence class. Now, $(p',q')$ has no choice but to be in the first three cases since the last case contains only non-minimal elements as we just argued.
\end{proof}

\begin{lemma} \label{lem:Gamma-computation}
Recall that $\Gamma(p,q) = p^2 + kq^2 - kpq$, Let ${\rm g.c.d.}(p,q) = d$.
\begin{itemize}
\item If $q \leq p/2$ and $p \leq kq/2$, then $\Gamma(p,q) < 0$ unless $k = 4$, $p = 2q$, in which case $\Gamma(p,q) = 0$ and $d = q$.
\item If $q > p$ or $p > kq$, then $\Gamma(p,q) > d^2$.
\item If $q = p$, then $\Gamma(p,q) = d^2$ and $d = p = q$ and if $p = kq$, then $\Gamma(p,q) = kq^2$ and $d = q$.
\end{itemize}
\end{lemma}

\begin{proof}
This is a straightforward computation and left to the reader.
\end{proof}

\begin{corollary} \label{cor:Gamma-stable}
Let $(p,q) \in \Z_{>1}^2$. Let $\Gamma(p,q) = p^2 + kq^2 - kpq$, Let ${\rm g.c.d.}(p,q) = d$. Then,
\begin{itemize}
\item $\alpha(p,q)$ is Schofield polystable if and only if $\Gamma(p,q) \leq d^2$.

\item $\alpha(p,q)$ is Schofield stable if and only if either $\Gamma(p,q) < 0$ or $\Gamma(p,q) \in \{0,1\}$ with $d = 1$.
\item $\alpha(p,q)$ is not Schofield semistable if and only if $\Gamma(p,q) > d^2$.
\end{itemize}

\end{corollary}

\begin{proof}
This follows from Proposition~\ref{prop:castling} and Lemma~\ref{lem:Gamma-computation}.
\end{proof}

\begin{proof} [Proof of Theorem~\ref{thm:same-dim}]
This follows from Corollary~\ref{cor:Gamma-stable} and Theorem~\ref{thm:realipca-to-sigma}.
\end{proof}

%%%%%%%%%%%%%%%%%%%%%%%%%%%%%%%%%%%%%%%%%%%%%%%%
\section{Polytopal characterization of Schofield stability}\label{sec:polytopal}
%%%%%%%%%%%%%%%%%%%%%%%%%%%%%%%%%%%%%%%%%%%%%%%%%
For a given quiver $Q$, and a dimension vector $\alpha \in \Z_{\geq 0}^{Q_0}$, recall the cone of effective weights: 

\begin{equation}
\Sigma(Q,\alpha) = \{\sigma \in \Z^{Q_0}\ |\ \SI(Q,\alpha)_{m\sigma} \neq 0 \text{ for some $m > 0$}\} \label{Eq:Sigma}
\end{equation}

There are several equivalent descriptions of the cone of effective weights. We write $\beta \hookrightarrow \alpha$ if a general representation of dimension $\alpha$ has a $\beta$-dimensional subrepresentation.
\begin{align}
\Sigma(Q,\alpha) &= \{\sigma \in \Z^{Q_0}\ |\ \SI(Q,\alpha)_{m\sigma} \neq 0 \text{ for some $m > 0$}\} \label{Eq:c-original}\\
&= \{\sigma \in \Z^{Q_0}\ |\ \SI(Q,\alpha)_{\sigma} \neq 0 \} \label{Eq:c-saturation} \\
&= \{\sigma \in \Z^{Q_0} \ |\ \alpha \text{ is $\sigma$-semistable}\}. \label{Eq:c-semistable} \\
& = \{\sigma \in \Z^{Q_0}\ |\ \sigma(\alpha) = 0 \text{ and } \sigma(\beta) \leq 0 \ \forall\ \beta \hookrightarrow \alpha\}. \label{Eq:c-king}
\end{align}

The equivalence of Equation~(\ref{Eq:c-original}) and Equation~(\ref{Eq:c-saturation}) follows from saturation, see \cite[Theorem~3]{DW-LR}. Given that, the equivalence of Equation~(\ref{Eq:c-original}) and Equation~(\ref{Eq:c-saturation}) with Equation~(\ref{Eq:c-semistable}) and Equation~(\ref{Eq:c-king}) is essentially due to King \cite{King}.

For a weight $\sigma \in \Z^{Q_0}$, we also define
$$
\overline{\Sigma}(Q,\sigma) := \{\alpha \in \Z_{\geq 0}^{Q_0}\ |\ \alpha \text{ is $\sigma$-semistable}\}.
$$

Suppose $\sigma \in \Z^{Q_0}$ such that $\overline{\Sigma}(Q,\sigma)$ is non-empty. Then there exists a dimension vector $\beta \in \Z_{\geq 0}^{Q_0}$ such that $\sigma(\gamma) = -\left<\gamma,\beta\right>$ for all $\gamma \in \R^{Q_0}$, see \cite[Theorem~1]{DW-LR}. Thus, we have
\begin{align*}
\overline{\Sigma}(Q,\sigma) & = \{\alpha \in \Z_{\geq 0}^{Q_0}\ |\ \alpha \text{ is $\sigma$-semistable}\} \\
&= \{\alpha \in \Z_{\geq 0}^{Q_0}\ |\ \alpha \text{ is $(-\left<-,\beta\right>$)-semistable}\} \\
&= \{\alpha \in \Z_{\geq 0}^{Q_0}\ |\ \beta \text{ is $\left<\alpha,-\right>$-semistable}\},
\end{align*}

The last equality follows from the reciprocity property, see \cite[Corollary~1]{DW-LR}. From the above, it is easy to see that $\overline{\Sigma}(Q,\sigma)$ is just $\Sigma(Q,\beta)$ up to a change of coordinates that we now make precise. Let $\zeta_Q: \R^{Q_0} \rightarrow \R^{Q_0}$ that is defined as follows: For $\delta \in \R^{Q_0}$, let $\zeta_Q(\delta) \in \R^{Q_0}$ be defined by 
\begin{align*}
\zeta_Q(\delta) \cdot \gamma = \sum_{x \in Q_0} \zeta_Q(\delta)_x \gamma_x = \left<\delta,\gamma\right> \text{ for } \gamma \in \R^{Q_0}.
\end{align*}
Then, by the above discussion, we conclude that 
\begin{align}
\overline{\Sigma}(Q,\sigma) = \zeta_Q (\Sigma(Q,\beta)).
\end{align}

In particular, this means that $\overline{\Sigma}(Q,\sigma)$ is also (the integer points of)
a convex polyhedral cone and we can write down inequalities that define it in terms of the subrepresentations of $\beta$. For the star quiver, we define a polytope that carries the same information as the cone. 

\begin{definition}[Polytope from slice of the cone]\label{def:polytope} Consider the polyhedral cone $\Sigma(\Delta[k], \beta)$ of effective weights for the star quiver $\Delta[k]$ with dimension vector $\beta$. We may instead study the polytope 
$$\Sigma_k(\beta) := \{x \in \Sigma(\Delta[k], \beta): x_{1} = 1\},$$ the set of weights with first coordinate equal to $1$. We have $\Sigma(\Delta[k], \beta) = \R_{\geq 0} \cdot  \Sigma_k(\beta)$ because the origin is the only element of $\Sigma(\Delta[k], \beta)$ with vanishing first coordinate.
\end{definition}

\subsection{A polynomial time algorithm for detecting $\sigma$-semistability for star quivers}

In order to give a polynomial time algorithm for detecting $\sigma$-semistability for star quivers, we reduce the problem to deciding if a (generalized) Littlewood-Richardson coefficient is non-zero. We first state the two main results of this section. 

\begin{theorem} \label{thm:ss-poly-time}
Consider the dimension vector $\alpha = (p,q_1,\dots,q_k)$ of the star quiver $\Delta[k]$ and let $\sigma = (\sigma_x,\sigma_{y_1},\dots,\sigma_{y_k})$. Then, we can decide if $\alpha$ is $\sigma$-semistable, i.e. if $\sigma \in \Sigma(\Delta[k], \alpha)$, in poly$(p,k)$-time.
\end{theorem}
Throughout this section, polynomial time will mean polynomial in $p$ and $k$.

%\VM{@Cole: I'm leaving it to you to fill in what exactly the run-time is polynomial in. Is it simply poly(p,n,k)-time? Or something stronger? Not just in the above theorem, but throughout this section} \CF{yes. I've updated accordingly - except I am confused about $n$ vs $k$. Shouldn't they be the same thing?}

\begin{proposition} \label{prop:ss-LR}
Consider the dimension vector $\alpha = (p,q_1,\dots,q_k)$ of the star quiver $\Delta[k]$ and let $\sigma = (\sigma_x,\sigma_{y_1},\dots,\sigma_{y_n})$ be a weight such that $\sigma_x \leq 0$ and $\sigma_{y_i} \geq 0$ for all $1 \leq i \leq k$. Then $\alpha$ is $\sigma$-semistable, i.e. $\sigma \in \Sigma(\Delta[k], \alpha)$, if and only if the Littlewood-Richardson coefficient $c_{(\sigma_{y_1})^{q_1},\dots (\sigma_{y_k})^{q
_k}}^{(-\sigma_x)^p}$ (as defined below) is non-zero.
\end{proposition}

First, we need to define Littlewood-Richardson coefficients rigorously. Consider the algebraic group $G = \GL_m = \GL(\C^m)$. Let $\Lambda = \Lambda_m = \{\lambda \in \Z^m\ |\ \lambda_1 \geq \lambda_2 \geq \dots \geq \lambda_m\}$. Irreducible representations of $\GL_m$ are naturally indexed by $\Lambda$, which is often called the dominant Weyl chamber. For each $\lambda \in \Lambda$, there is an irreducible representation $S_\lambda(\C^m)$ of $\GL_m$. We call $\lambda$ the highest weight for $S_\lambda(\C^m)$.

%The irreducible representations (over $\C$) of $G$ are indexed by their highest weights. For  $\lambda \in \Z^{\dim V}$ . For each partition $\lambda = \lambda_1 \lambda_2 \dots \lambda_k$, we denote by $S_\lambda(V)$ the corresponding irreducible representation, see \cite{Fulton} for a definition. We use the convention in \cite{Fulton}, so $S_{(n)}(V) = \Sym^n(V)$, the $n^{th}$ symmetric power and $S_{1^n}(V) = \bigwedge^n(V)$ is the $n^{th}$ exterior power.

For $\lambda^{(1)}, \dots, \lambda^{(k)} \in \Lambda$, we decompose the tensor product as a direct sum of irreducibles
\begin{equation} \label{Eq:def-LR}
S_{\lambda^{(1)}}(\C^m) \otimes S_{\lambda^{(2)}}(\C^m) \otimes \dots \otimes S_{\lambda^{(k)}} (\C^m) = \bigoplus_{\mu} S_\mu(\C^m)^{\oplus \left(c_{\lambda^{(1)}, \dots, \lambda^{(k)}}^\mu\right)}. 
\end{equation}

The multiplicities $c_{\lambda^{(1)}, \dots, \lambda^{(k)}}^\mu$ are called (generalized) Littlewood-Richardson coefficients. In literature, often Littlewood-Richardson coefficients refer to the case when $k = 2$, i.e., the multiplicites obtained by decomposing a tensor product of two irreducibles.

For $\lambda = (\lambda_1,\dots,\lambda_m) \in \Lambda$, we define $\widehat{\lambda} = (-\lambda_m,-\lambda_{m-1},\dots,-\lambda_1) \in \Lambda$. And we note that
\begin{align} \label{Eq:dual-Schur}
   (S_\lambda(\C^m))^* = S_\lambda((\C^m)^*) = S_{\widehat{\lambda}}(\C^m) 
\end{align}

%\VM{Ok, this takes too much effort to introduce}

For the product of general linear groups $\GL(\alpha) = \GL_p \times \GL_{q_1} \times \dots \times \GL_{q_k}$, the irreducible representations are all of the form 
$$
E_{\lambda^{(x)}, \lambda^{(y_1)}, \dots, \lambda^{(y_k)}} = S_{\lambda^{(x)}}(\C^p) \otimes \bigotimes_i S_{\lambda^{(y_i)}}(\C^{q_i}).
$$

One observes that $\SI(\Delta[k],\alpha)_\sigma$ is precisely the isotypic component in $\C[\Rep(\Delta[k],\alpha)]$ corresponding to $E_{(\sigma_x)^p, (\sigma_{y_1})^{q_1}, \dots, (\sigma_{y_k})^{q_k}}$. To do this, let us decompose $\C[\Rep(\Delta[k],\alpha)]$ into irreducible representations for the action of $\GL(\alpha)$. 

We define $\Lambda_m^+ = \Lambda_m \cap \Z_{\geq 0}^m$. We observe that $\Rep(\Delta[k],\alpha) = \bigoplus_{i=1}^k (\C^p) \otimes (\C^{q_i})^*$. Thus, we have 
\begin{align*}
\C[\Rep(\Delta[k],\alpha)] &= {\rm Sym} (\Rep(\Delta[k],\alpha)^*)\\
&= \bigotimes_{i=1}^k {\rm Sym} ((\C^p)^* \otimes \C^{q_i}))\\
& = \bigoplus_{\lambda^{(y_i)} \in \Lambda_{y_i}^+} \left( \bigotimes_{i=1}^k S_{\lambda^{(y_i)}}(\C^p)^* \otimes S_{\lambda^{(y_i)}}(\C^{q_i})\right) \\
&= \bigoplus_{\lambda^{(y_i)} \in \Lambda_{y_i}^+} \left( \bigotimes_{i=1}^k S_{\lambda^{(y_i)}}(\C^p)^*\right) \otimes \left(\bigotimes_{i=1}^k S_{\lambda^{(y_i)}}(\C^{q_i})\right) \\
& = \bigoplus_{\lambda^{(y_i)} \in \Lambda_{y_i}^+} \left( \bigoplus_{\mu} S_{\widehat{\mu}}(\C^p)^{\oplus \left(c_{\lambda^{(y_1)}, \dots, \lambda^{(y_k)}}^\mu\right)} \right) \otimes \left(\bigotimes_{i=1}^k S_{\lambda^{(y_i)}}(\C^{q_i})\right) \\
& = \bigoplus_{\mu, \lambda^{(y_i)} \in \Lambda_{y_i}^+} (E_{\widehat{\mu}, \lambda^{(y_1)}, \dots, \lambda^{(y_k)}})^{\oplus \left( c_{\lambda^{(y_1)}, \dots, \lambda^{(y_k)}}^\mu \right) }
\end{align*}

The first equality follows from the fact that polynomial functions $\C[W]$ on any vector space $W$ is the same as ${\rm Sym}(W^*)$. The second equality is clear. The third equality follows from Cauchy's identity and the fourth equality is simply a regrouping of terms. The fifth equality is a consequence of the definition of Littlewood-Richardson coefficients, i.e., Equation~\ref{Eq:def-LR}, and Equation~\ref{Eq:dual-Schur}. The final equality is simply a rearrangement of terms.

\begin{proof} [Proof of Proposition~\ref{prop:ss-LR}]
From the above discussion, we observe that 
 $\SI(\Delta[k],\alpha)_\sigma$ is non-zero if and only if $c_{(\sigma_{y_1})^{q_1}, \dots, (\sigma_{y_k})^{q_k}}^{(-\sigma_x)^p}$ is non-zero because $\widehat{(-\sigma_x)^p} =(\sigma_x)^p$. Since $\SI(\Delta[k],\alpha)_\sigma \neq 0$ precisely when $\alpha$ is $\sigma$-semistable, we have the required conclusion.
\end{proof}

Now, we turn to proving Theorem~\ref{thm:ss-poly-time}. 

\begin{proof}[Proof of \cref{thm:ss-poly-time}]
From Proposition~\ref{prop:ss-LR}, we see that it suffices to show that the vanishing of (generalized) Littlewood-Richardson coefficients can be decided in poly-time. In the case $k = 2$, i.e., $c_{\lambda, \mu}^\nu$, a (strongly) polynomial time algorithm for deciding its vanishing was shown in \cite{mulmuley2012geometric} (another algorithm was given in \cite{bi-LR}). One can simply adapt this to get a poly-time algorithm for computing $c_{\lambda^{(1)}, \dots, \lambda^{(k)}}^\nu.$

To describe this, we must first describe the algorithm of \cite{mulmuley2012arxiv}. The authors show $c_{\lambda, \mu}^\nu$ vanishes if and only if the polyhedron $P_{\lambda, \mu}^\nu := \{x \in \QQ^{m_2}: A x \leq b_{\lambda, \mu}^\nu\}$ is empty, where $b_{\lambda, \mu}^\nu \in \QQ^{m_1}$ is a vector whose entries are homogeneous linear forms in $\lambda, \mu, \nu$ with coefficients in $\{0, -1, 1\}$ and $A \in \Mat_{m_1, m_2}(\{0, -1, 1\})$ is a matrix depending only on $n$ and not $\lambda, \mu, \nu$. The numbers $m_1, m_2$ are both $O(n^2)$. Thus deciding the vanishing of $c_{\lambda, \mu}^\nu$ can be decided in strongly polynomial time by the existence of strongly polynomial time algorithms for linear programs with coefficients in $\{0,-1,1\}$ \cite{tardos1986strongly}.

We claim that a similar statement holds for $c_{\lambda^{(1)}, \dots, \lambda^{(k)}}^\mu,$ namely that there is a vector $b_{\lambda^{(1)}, \dots, \lambda^{(k)}}^\nu \in \QQ^{m_1(k)}$ whose entries are linear forms in $\lambda^{(1)}, \dots, \lambda^{(k)}, \nu$ with coefficients in $\{0,-1,1\}$ and $A_k \in \Mat_{m_1(k), m_2(k)}(\{0, -1, 1\})$ depending only on $n$ and $k$ and not $\lambda^{(1)}, \dots, \lambda^{(k)}, \nu$ such that $c_{\lambda^{(1)}, \dots, \lambda^{(k)}}^\mu = 0$ if and only if the polyhedron $P_{\lambda^{(1)}, \dots, \lambda^{(k)}}^\nu := \{x \in \QQ^{m_2(k)}: A_k x \leq b_{\lambda^{(1)}, \dots, \lambda^{(k)}}^\nu\}$ is empty. Moreover, $m_1(k), m_2(k) = O(k n^2)$. This will imply that the vanishing of $c_{\lambda^{(1)}, \dots, \lambda^{(k)}}^\nu$ can also be decided in strongly polynomial time. We prove the claim by induction on $k\geq 2$; the base case is already established with $m_1 = m_1(2), m_2 = m_2(2), A_2 = A$. 

Now suppose $k \geq 3$. By saturation, $c_{\lambda^{(1)}, \dots, \lambda^{(k)}}^\nu > 0$ if and only if there exists $t \in \N$ such that $c_{t\lambda^{(1)}, \dots, t\lambda^{(k)}}^{t\nu} > 0$. By definition, $c_{t\lambda^{(1)}, \dots, t\lambda^{(k)}}^{t\nu} > 0$ if and only if there some partition $\alpha$ such that $c_{t\lambda^{(1)}, \dots, t\lambda^{(k-1)}}^\alpha > 0$ and $c_{t\lambda^{(k)}, \alpha}^{t\nu} > 0.$ That is, $c_{\lambda^{(1)}, \dots, \lambda^{(k)}}^\nu > 0$ if and only if there exists $t \in \N$, a non-increasing vector $\alpha \in \Z_{ \geq 0}^n$, and $x_1 \in \QQ^{m_2(k-1)}, x_2 \in \QQ^{m_2}$ such that $A_{k-1} x_1 \leq b_{t\lambda^{(1)}, \dots, t\lambda^{(k-1)}}^\alpha$ and $A x_2 \leq b_{t\lambda^{(k)}, \alpha}^{t\nu}.$ As the entries of $b_{\lambda^{(1)}, \dots, \lambda^{(k-1)}}^\alpha, b_{\lambda^{(k)}, \alpha}^\nu$ are homogeneous linear forms in the partitions, this is true if and only if there exists $t \in \N$, a non-increasing vector $\alpha \in \Z_{ \geq 0}^n$, and $x_1 \in \QQ^{m_2(k-1)}, x_2 \in \QQ^{m_2}$ such that $A_{k-1} x_1 \leq b_{\lambda^{(1)}, \dots, \lambda^{(k-1)}}^{\alpha/t}$ and $A x_2 \leq b_{\lambda^{(k)}, \alpha/t}^\nu.$ This is equivalent to the existence of a non-increasing vector $\alpha \in \QQ_{ \geq 0}^n$ and $x_1 \in \QQ^{m_2(k-1)}, x_2 \in \QQ^{m_2}$ such that $A_{k-1} x_1 \leq b_{\lambda^{(1)}, \dots, \lambda^{(k-1)}}^{\alpha}$ and $A x_2 \leq b_{\lambda^{(k)}, \alpha}^\nu.$ This is then equivalent to the feasibility of a larger linear program, namely the nonemptiness of the polyhedron 
\begin{align*}
P' = \{(\alpha, x_1, x_2) \in \QQ^{n + m_2(k-1) + m_2}: A_{k-1} x_1 &\leq b_{\lambda^{(1)}, \dots, \lambda^{(k-1)}}^{\alpha}, \\
A x_2 &\leq b_{\lambda^{(k)}, \alpha}^\nu,\\
 \alpha_{i + 1} &\leq \alpha_{i}\; \forall i \in [n-1]\\
 \alpha_i & \geq 0;  \forall i \in [n]
 \}.\end{align*}
One checks that the matrix $A_k$ and vector $b_{\lambda^{(1)}, \dots, \lambda^{(k)}}^\nu\}$ such that 
$$P' = \{y \in \QQ^{ m_2(k)}: A_k y \leq b_{\lambda^{(1)}, \dots, \lambda^{(k)}}^\nu\}$$ have the desired properties. Namely, we may $m_2(k) = m_2(k-1) + m_2 + n$ and $m_1(k) = m_1(k-1) + m_1 + n^2 + n$ which, by induction, are both $O(k n^2)$. Moreover, because $A_{k-1}$ and $A$ have entries in $\{0, -1, 1\}$ and $b_{\lambda^{(1)}, \dots, \lambda^{(k-1)}}^\alpha, b_{\lambda^{(k)}, \alpha}^\nu$ are homogeneous linear forms with coefficients in $\{0, -1, 1\}$, $A_k$ has entries in $\{0,-1,1\}$ and $b_{\lambda^{(1)}, \dots, \lambda^{(k)}}^\nu$ consists of homogeneous, linear forms with coefficients in $\{0,-1,1\}$. This completes the proof. \end{proof}
As a corollary, the polyhedra $\Sigma_k(\beta), \Sigma(\Delta[k], \beta)$ of \cref{Eq:Sigma}, \cref{def:polytope} are \emph{well-described}: 
\begin{corollary}\label{cor:opt-oracle}The polyhedra $\Sigma(\Delta[k],\beta), \Sigma_k(\beta)$ are \emph{well-described}, i.e. it is given by a finite collection of rational inequalities, each of polynomial encoding length. Moreover, there is a polynomial time algorithm to optimize linear functions over $\Sigma_k(\beta)$.
\end{corollary}
\begin{proof} Consider the Horn polytope $P$, the rational cone generated by partitions $(\lambda^{(1)}, \dots, \lambda^{(k)}, \nu)$ such that $c_{\lambda^{(1)}, \dots, \lambda^{(k)}}^\nu > 0$. By \cref{prop:ss-LR}, the cone $\Sigma(\Delta[k], \beta)$ is a preimage of $P$ under a linear map with Boolean matrix entries in the coordinate bases. $\Sigma_k(\beta)$ an intersection of $\Sigma(\Delta[k], \beta)$ with the coordinate plane $x_1 = 1$, and hence is well-described if $\Sigma_k(\beta)$ is. Therefore it is enough to show that $P$ is well-described. We now argue that this is the case.

Indeed, for rational $(\lambda^{(1)}, \dots, \lambda^{(k)}, \nu)$, we saw in the proof of \cref{thm:ss-poly-time} that there are vectors whose entries are linear forms in $\lambda^{(1)}, \dots, \lambda^{(k)}, \nu$ with coefficients in $\{0,-1,1\}$ and a Boolean matrix $A_k$ such that $y = (\lambda^{(1)}, \dots, \lambda^{(k)}, \nu) \in P$ if and only if the polyhedron $P_{\lambda^{(1)}, \dots, \lambda^{(k)}}^\nu := \{x \in \QQ^{m_2(k)}: A_k x \leq b_{\lambda^{(1)}, \dots, \lambda^{(k)}}^\nu\}$ is nonempty. This implies $P$ is a projection of a well-described polyhedron $Q$, and hence $P$ is well-described. This can be seen in several ways, e.g. Fourier-Motzkin elimination. 

It remains to be seen why there is a polynomial time algorithm to optimize linear functions over $\Sigma_k(\beta)$. By Theorem 6.4.1 in \cite{GLS-book}, strong optimization and separation are polynomially equivalent for well-described polyhedra, so it is enough to show that $\Sigma_k(\beta)$ has a polynomial time separation oracle.
%, which in turn follows from the existence of a polynomial time  separation oracle for $\Sigma(\Delta[k], \beta)$. 
By \cref{prop:ss-LR}, the cone $\Sigma_k(\beta)$ is a preimage of the slice $P':=P\cap\{\sigma:x_1 = 1\}$ of the Horn polytope under a linear map with Boolean matrix entries in the coordinate bases. 

Thus it suffices to show that $P'$ has a strong separation oracle. Recall that $P$ is a coordinate projection of $Q$, and that both are well-described. Thus $P'$ is a coordinate projection of a well-described coordinate slice $Q'$ of $Q$. As the inequalities of $Q$ and hence $Q'$ are given explicitly, $Q'$ has a strong separation oracle and hence a strong optimization oracle. As $P'$ is a projection of $Q'$, $P'$ has a strong optimization oracle as well. Hence $P'$ also has a strong separation oracle. This completes the proof.\end{proof}

\subsection{The Derksen-Weyman characterization of $\sigma$-stability}

The set of $\sigma$-semistable dimension vectors, i.e. $\overline{\Sigma}(Q,\sigma)$, is (the integer points of) a rational convex polyhedral cone and membership of dimension vectors in $\overline{\Sigma}(Q,\sigma)$ can be solved in polynomial time by Theorem~\ref{thm:ss-poly-time}. On the other hand, $\sigma$-stability of a dimension vector $\alpha$ is more subtle to understand and detect. The subset of dimension vectors in $\overline{\Sigma}(Q,\sigma)$ that are $\sigma$-stable usually lacks convexity. We say a dimension vector $\alpha \in \Z^{Q_0}$ is indivisible if ${\rm g.c.d.}(\alpha_x:x \in Q_0) = 1$.

\begin{lemma}
Let $Q$ be a quiver, $\alpha$ a dimension vector and $\sigma$ a weight. Suppose $\alpha$ generates an extremal ray of $\overline{\Sigma}(Q,\sigma)$ (in particular $\alpha$ is indivisible). Then 
$\alpha$ is $\sigma$-stable.
\end{lemma}

\begin{proof}
By definition of $\overline{\Sigma}(Q,\sigma)$, $\alpha$ is $\sigma$-semistable. There is a concept of $\sigma$-stable decomposition, i.e., for any dimension vector $\beta$ which is $\sigma$-semistable, we write $\beta = \gamma_1 \dotplus \gamma_2 \dotplus \dots \dotplus \gamma_k$ if a general representation of dimension $\beta$ has a Jordan-H\"older filtration (in the category of $\sigma$-semistable representations) whose factors are of dimensions $\gamma_1,\gamma_2,\dots,\gamma_k$ (in some order), see \cite[Section~3]{DW-comb}. Each of the $\gamma_i$ are $\sigma$-stable, hence $\sigma$-semistable and so $\gamma_i \in \overline{\Sigma}(Q,\sigma)$.

Now, since $\alpha$ generates an extremal ray, it cannot be written non-trivially as a sum of dimension vectors in $\overline{\Sigma}(Q,\sigma)$. Thus, its $\sigma$-stable decomposition must be trivial, i.e., $\alpha$ is $\sigma$-stable.
\end{proof}

The locus of $\sigma$-stable dimension vectors in $\overline{\Sigma}(Q,\sigma)$ apart from the extremal rays is more complicated, but is captured by the result below due to Derksen and Weyman, see \cite[Theorem~6.4]{DW-comb}. Suppose we have a sequence of dimension vectors $\underline{\delta} = \delta_1,\dots,\delta_t$ (all $\delta_i$ distinct), we define $G(\underline{\delta})$ to be the directed graph with vertex set $\{1,2,\dots,t\}$ and an edge from $i$ to $j$ if $\left<\delta_i,\delta_j\right> < 0$.  

\begin{theorem} [Derksen-Weyman] \label{thm:DW-stability}
Suppose $\alpha \in \overline{\Sigma}(Q,\sigma)$. Then there exists a linearly independent sequence $\underline{\delta} = \delta_1,\dots,\delta_t$ of dimension vectors that generate extremal vectors of $\overline{\Sigma}(Q,\sigma)$ such that $\alpha$ is a positive rational combination of $\delta_1,\dots,\delta_t$. The dimension vector $\alpha$ is $\sigma$-stable if and only if
\begin{enumerate}
    \item $\alpha = \delta_1$ and $\delta_1$ is a real Schur root.
    \item or $\left< \delta_i,\alpha\right> \leq 0$ and $\left<\alpha,\delta_i\right> \leq 0$ for all $i$ and $G(\underline{\delta})$ is path-connected and $\alpha$ is indivisible if it is isotropic.
\end{enumerate}
\end{theorem}

For fixed $\delta_1,\dots, \delta_t$, the condition for $\alpha$ to be proportional to some element of the semistable locus is polyhedral. Thus we have the following. 
\begin{corollary}\label{cor:union-cones} Let $S$ be the locus of $\sigma$-stable dimension vectors in $\overline{\Sigma}(Q,\sigma)$. Then $\R_+ S = \{a s: a \in \R, s \in S\}$ is a finite union of convex cones.
\end{corollary}

\subsection{A polynomial time algorithm for detecting $\sigma$-stability for star quivers}

We turn to proving the following theorem:

\begin{theorem} \label{thm:stable-poly-time}
Let $\alpha$ be a dimension vector and $\sigma$ be a weight for the quiver $\Delta[k]$. Then, in polynomial time, we can decide if $\alpha$ is $\sigma$-stable.
\end{theorem}

It is well-known that, given an optimization oracle for a polytope, one can efficiently express interior points as convex combinations of vertices.
\begin{theorem}[6.5.11 in \cite{GLS-book}]\label{thm:mem-vectors} For a well-described polytope $(P; n, \Phi)$ given by a strong optimization oracle, there is a polynomial time algorithm to express a rational interior point $p$ as a convex combination of vertices $v_1,\dots,v_t$ of $P$.
\end{theorem}

%\begin{theorem} \label{thm:mem-vectors}Given a membership oracle for a rational convex polytope, for any interior point $p$ we can efficiently express $p$ as a positive rational combination of vertices $v_1,\dots,v_t$ of the polytope.\end{theorem}

This theorem applies to our case because the polyhedron $\Sigma(Q,\beta)$ is given by a strong optimization oracle by \cref{cor:opt-oracle}:

\begin{corollary}
Consider the quiver $\Delta[k]$ and let $\alpha$ be a dimension vector and $\sigma$ a weight such that $\alpha \in \overline{\Sigma}(Q,\sigma)$. Then, in polynomial time, we can find linearly independent integral vectors $\delta_1,\dots,\delta_t$ that generate extremal vectors of $\overline{\Sigma}(Q,\sigma)$ such that $\alpha$ is a positive rational combination of $\delta_1,\dots,\delta_t$.
\end{corollary}

%\begin{proof}
%We have a membership oracle for $\overline{\Sigma}(Q,\sigma)$ by Theorem~\ref{thm:ss-poly-time}. 
%We have an optimization oracle for $\overline{\Sigma}(Q,\sigma)$ by Corollary~\ref{cor:opt-oracle}.
%Thus, by Theorem~\ref{thm:mem-vectors}, we get the required conclusion. \end{proof}

\begin{proof} [Proof of Theorem~\ref{thm:stable-poly-time}]
We first check if $\alpha$ is $\sigma$-semistable (can be done efficiently by Theorem~\ref{thm:ss-poly-time}). If $\alpha$ is $\sigma$-semistable, then we can find efficiently $\delta_1,\dots,\delta_k$ satisfying the hypothesis in Theorem~\ref{thm:DW-stability}. Now, it is straightforward and efficient to check whether condition $(1)$ or condition $(2)$ is satisfied, so we can decide whether $\alpha$ is $\sigma$-stable efficiently. 
\end{proof}
Now we have proved all we need for \cref{thm:polyhedral}. 
\begin{proof}[Proof of \cref{thm:polyhedral}] The first statement follows from \cref{cor:ipca-Malpha}, the second from the discussion at the beginning of this section, and the third from \cref{cor:union-cones} and \cref{thm:stable-poly-time}.\end{proof}
\FloatBarrier

\subsection{Randomized algorithm for generic stability of star quivers}

Here we include an algorithm to decide the generic stability of star quivers, and hence generic unique existence of iPCA. The algorithm takes advantage of the fact that semistability and polystability are generically equivalent for star quivers. First we decide semistability, and hence polystability, by checking whether the likelihood function is bounded below. Next we distinguish polystability from stability by checking if there is a stabilizer subgroup of positive dimension. For the latter calculation, we need only check that there is no nontrivial element of the Lie algebra stabilizing the element in question. Because the Lie algebra action is linear, we can check this by evaluating the rank of the Lie algebra action as a linear map. Let 
$$ \pi: \mathfrak{gl}(\alpha) = \Mat_{p,p} \oplus \bigoplus_{i=1}^k \Mat_{q_i,q_i}  \to \Rep(\Delta[k], \alpha)) = \bigoplus_{i=1}^k \Mat_{p,q_i}$$
denote the Lie algebra action of $\mathfrak{gl}(\alpha)$ on some element $B \in \bigoplus_{i=1}^k \Mat_{p,q_i}$. We can compute the Lie algebra action as follows:
\begin{equation}\pi(H_0, H_1, \dots, H_k)B = (H_0 B_1 - B_1 H_1, \dots, H_0 B_k - B_k H_k).\label{eqn:pi}
\end{equation} 
Any element of the form $H_0 = t I_{p}, H_i = t I_{q_i}$ will be in the stabilizer, but any other elements are non-trivial.

\begin{Algorithm}
\begin{description}
\item[\hspace{.2cm}\textbf{Input}] A dimension vector $\alpha = (p, q_1, \dots, q_k)$. 

%A representation of $\Delta_k$ with dimension vector $(p, q_1, \dots, q_k)$, given as  matrices $B_1, \dots, B_k$ 

\item[\hspace{.2cm}\textbf{Output}] \textbf{Yes} if $\alpha$ is generically stable, \textbf{No} otherwise.

\item[\hspace{.2cm}\textbf{Algorithm}]
\end{description}
\begin{enumerate}
\item\label{it:sample} For each $i \in \{1, \dots, k\}$, sample a $p\times q_i$ matrices $B_i$ with i.i.d. complex Gaussian entries. Think of $B = (B_1, \dots B_k)$ as a representation of the start quiver $\Delta[k]$.
\item Check if $B = (B_1, \dots, B_k)$ is semistable using e.g. \cite{GGOW16}. If not, output \textbf{No}. Otherwise continue to the next step.
\item Compute the Lie algebra action $\pi$ of $\mathfrak{gl}(\alpha)$ on $B$ as in \cref{eqn:pi}.
If the dimension of the kernel of $\pi$ is at most $1$, output \textbf{Yes.} Otherwise output \textbf{No.}
\end{enumerate}
\caption{Randomized algorithm to decide generic stability of star quivers.}\label{alg:star-stable}
\end{Algorithm}

%The main mechanism of the algorithm is that semistability and polystability are generically equivalent, dn

\appendix

\FloatBarrier
\section{From complex iPCA to real iPCA} \label{app: algebraic groups over R} 

A complex (affine) variety $X$ with an $\R$-structure (i.e., an $\R$-subalgebra $\R[X] \subseteq \C[X]$ such that $\R[X] \otimes_\R \C = \C[X]$\footnote{For any complex (affine) variety $Y$, we denote by $\C[Y]$ its coordinate ring}) is called an (affine) $\R$-variety. As a technical point, we identify a variety $X$ with its complex points $X_\C$ (which can be viewed as algebra morphisms $\C[X] \rightarrow \C$). A morphism $f:X \rightarrow Y$ of (affine) varieties is equivalent to a map on the coordinate rings $f^*: \C[Y] \rightarrow \C[X]$. The morphism $f$ is said to be defined over $\R$ if $f^*(\R[Y]) \subseteq \R[X]$. The real points $X_\R$ as the points in $X_\C$ (i.e., the algebra morphisms $\C[X] \rightarrow \C$) that are defined over $\R$. A complex algebraic group $G$ is called an $\R$-group if it is an (affine) $\R$-variety such that the multiplication map and inverse map are defined over $\R$, and its real points $G_\R$ is an algebraic group over $\R$.

\begin{remark}
We only deal with affine varieties in this paper, so we will drop the prefix affine from here on.
\end{remark}

For a representation $V$ of an algebraic group $G$ (over $K = \R$ or $\C$), we say that $V$ is generically $G$-semistable (resp. $G$-polystable, $G$-stable) if there is a Zariski open dense subset of $V$ that consists of $G$-semistable points (resp. $G$-polystable points, $G$-stable points). The following is a very important result.

\begin{proposition} \label{gen.stable.transfer}
Let $G$ be a connected reductive $\R$-group. Let $V$ be a rational representation of $G$ that is defined over $\R$. Then $V$ is generically $G$-semistable (resp. $G$-polystable, $G$-stable) if and only if $V_\R$ is generically $G_\R$-semistable (resp. $G_\R$-polystable, $G_\R$-stable).
\end{proposition}

\begin{proof}
See \cite[Proposition~2.23]{derksen2020maximum}.
\end{proof}

Thus, as far as (generic) semistability/polystability/stability is concerned, there is no difference whether you consider the (actions corresponding to the) real model $\mathcal{M}_{\alpha,\R}$ or the complex model $\mathcal{M}_{\alpha,\C}$. However, there is still one subtle issue, namely that stability is not equivalent to existence of a unique MLE for real Gaussian groups (see Theorem~\ref{theo:AKRS}). However, we leverage Lemma~\ref{lem:rect-AKRS} instead. First, a lemma:

\begin{lemma} \label{lem:transfer-gen-indec}
Let $\alpha = (p,q_1,\dots,q_r)$. Suppose a generic $V \in \Rep(\Delta[k], \alpha)_\R$ is indecomposable (over $\R$), then a generic $V \in \Rep(\Delta[k], \alpha)_\C$ is indecomposable (over $\C$).
\end{lemma}

\begin{proof}
We will prove this by contradiction. Suppose a generic $V \in \Rep(\Delta[k], \alpha)_\C$ is not indecomposable. Then, its canonical decomposition must be non-trivial, say it is
$$
\alpha = \beta_1 \oplus \beta_2 \oplus \dots \oplus \beta_k.
$$

This means that the natural map $\psi: \GL(\alpha)_\C \times \prod_{i=1}^k \Rep(\Delta[k], \beta_i)_\C =: X \rightarrow \Rep(\Delta[k], \alpha)_\C =: Y$ is a dominant morphism of varieties. Observe that both $X$ and $Y$ are varieties that are defined over $\R$ in a natural way and that the $\R$-points $X_\R$ and $Y_\R$ are Zariski-dense in $X$ and $Y$ respectively. Since $X_\R$ is Zariski-dense and $\psi$ is dominant, we conclude that $\psi(X_\R)$ is Zariski-dense in $Y$ and hence Zariski-dense in $Y_\R$. Now, $\psi: X_\R \rightarrow Y_\R$ is a morphism of real algebraic varieties and hence the image $\psi(X_\R)$ is a semi-algebraic set. Since it is dense, it must be a full dimensional semi-algebraic set. 

Observe that $\psi(X_\R)$ consists of representations that are decomposable (over $\R$). Thus, we have a full dimensional semi-algebraic set, i.e., $\psi(X_\R)$ which consists of representations that are decomposable. Thus, we cannot have a Zariski open dense subset of $Y_\R$ consisting of indecomposable representations.
\end{proof}

\begin{lemma} \label{lem:dec-compact}
Let $V \in \Rep(Q,\alpha)_\R$. If $V$ is decomposable over $\R$, then the image of the stabilizer $H = \rho_\alpha((\GL(\alpha)_\sigma(\R))_V)$ is not compact. 
\end{lemma}

\begin{proof}
Suppose $V = U \oplus W$ is a decomposition of $V$ over $\R$. In other words, for each $x \in Q_0$, we have a splitting $V(x) = \R^{\alpha_x}  = U(x) \oplus W(x)$ and for each arrow $a \in Q_1$, $V(a)$ maps $U(ta) \rightarrow U(ha)$ and maps $W(ta) \rightarrow W(ha)$. Suppose $\dim(U) = \beta$ and $\dim(W) = \gamma$.

Let $\lambda = 2$ and let $\mu \in \R_{>0}$ be such that $\lambda^{\sigma \cdot \beta} \mu^{\sigma \cdot \gamma} = 1$. We need to argue that such a $\mu$ exists. Let $\nu_1= \sigma \cdot \beta$ and $\nu_2 = \sigma \cdot \gamma $. Since $\sigma \cdot \alpha = 0$, either both $\nu_1$ and $\nu_2$ are zero or both non-zero. If both are zero, then any choice of $\mu$ works. If both are non-zero, then there is a unique choice of $\mu$. 

Consider the linear transformation $g_x \in \GL_{\alpha_x}$ that acts on $U(x)$ by the scalar $\lambda$ and on $W(x)$ by the scalar $\mu$. It is easy to check that $g = (g_x)_{x \in Q_0} \in \GL(\alpha)_\sigma$ because $\prod_{x \in Q_0} \det(g_x)^{\sigma_x} = \lambda^{\sigma \cdot \beta} \mu^{\sigma \cdot \gamma} = 1$.

It is straightforward to see that $\rho_\alpha(g)$ has eigenvalues with absolute value $\neq 1$. Thus the sequence $\{\rho_\alpha(g^m)\}_{m \in \Z_{>0}}$ is a sequence in $H$ with no convergent subsequence. Thus $H$ cannot be compact.
\end{proof}

Before we prove the theorem, we will need the following lemma. For an algebraic group $G$ (defined over $K = \R$ or $\C$), we denote by $G^\circ$ its identity component, i.e., the connected component (in the Euclidean topology) of $G$ that contains the identity element. The following lemma was proven for $K = \C$ in \cite[Lemma~2.17]{derksen2020maximum}, but it is also true for $K = \R$ as we will now show.

\begin{lemma} \label{lem:conn-component-stability}
Let $K = \R$ or $\C$ denote the ground field. Let $G$ be an algebraic group and $V$ a rational representation. Let $G^\circ$ denote the identity component of $G$. Then for $v \in V$, $v$ is $G$-semistable/polystable/stable if and only if $v$ is $G^\circ$-semistable/polystable/stable.
\end{lemma}

\begin{proof}
First, we note that $G^\circ$ is a subgroup of finite index in $G$, equivalently that $G$ has finitely many components. For $K = \C$, this is well known because connected components in the Euclidean topology agree with the connected components in the Zariski topology, the latter of which is always finite by Noetherianity. For $K = \R$, this follows from Whitney \cite[Theorem~3]{Whitney} which proves that any real algebraic variety has finitely many topological components.

There is one subtle issue in the case of $K = \R$. We note that $G^\circ$ may not be an algebraic group over $\R$! However, it is still a Lie group and all the arguments used here should be taken in the framework of Lie theory. We only defined stability notions for algebraic groups, but their definitions extend easily to any group action, in particular for actions of Lie groups. 

Suppose $v$ is $G$-semistable, then it is clearly $G^\circ$-semistable. Conversely, suppose $v$ is not $G$-semistable, then there is a sequence $\{g_m\}$ such that $\lim_{m \to \infty} g_m v = 0$. Since $G$ has finitely many components, we can extract a subsequence $\{h_m\}$ with the same property but is contained in one connected component, say $j G^\circ$ for some $j \in G$. Then, we see that $\{j^{-1}h_m\}$ is a sequence in $G^\circ$ such that $\lim_{m \to \infty} j^{-1} h_m v = 0$. Hence $v$ is $G^\circ$-unstable.

The arguments for polystability and stability are the same as in \cite[Lemma~2.17]{derksen2020maximum}, but we recall them here for completeness. For polystability, we observe that the $G$-orbit is a disjoint finite union of $G^\circ$-orbits each of which forms a connected component of the $G$-orbit. Hence the $G$-orbit of $v$ is closed if and only if the $G^\circ$ orbit of $v$ is closed.

For stability, we need to observe that (by the orbit-stabilizer theorem) $\dim(G_v) = \dim((G^\circ)_v)$ because their orbits have the same dimensions. Further suppose the action of $G$ on $V$ is given by $\rho: G \rightarrow \GL(V)$. Then, we observe that the kernel of $\rho$ (denoted $\Delta$) and the kernel of $\rho|_{G^\circ}$ (i.e., $\Delta \cap G^\circ$) have the same dimension because they both have the same Lie algebra. Thus, $\dim(G_v) = \dim(\Delta)$ if and only if $\dim((G^\circ)_v) = \dim(\Delta \cap G^\circ)$. 
\end{proof}

Now, we have all the ingredients need to prove Theorem~\ref{thm:realipca-to-sigma}.

\begin{proof} [Proof of Theorem~\ref{thm:realipca-to-sigma}]
We observe that $\GL(\alpha)_\sigma$ is a reductive $\R$-group and $\Rep(\Delta[k], \alpha)) = \oplus_{i=1}^k \mat_{p,q_i}$ is an $\R$-variety in a natural way. Since $\GL(\alpha)_\sigma$ is not a connected group, which prevents us from applying Proposition~\ref{gen.stable.transfer} directly, which is precisely why we prove Lemma~\ref{lem:conn-component-stability}.

Let $\mu = (\mu(x), \mu(y_1), \dots, \mu(y_k))$ be the unique indivisible integral weight that is a positive multiple of $\sigma$. In other words, $\sigma = d \mu$ for some $d \in \Z_{\geq 1}$ and ${\rm g.c.d.} (\mu(x), \mu(y_1),\dots,\mu(y_k)) = 1$. First, we observe that $\GL(\alpha)_\mu$ is connected, see \cite[Remark~2.19]{derksen2020maximum}. Consider the map $\mu: \GL(\alpha)_\sigma \rightarrow \C^\times$ given by $\mu((g_x,g_{y_1}, \dots, g_{y_k})) = \det(g_x)^{\mu(x)} \cdot \prod_{i=1}^k \det(g_{y_i})^{\mu(y_i)}$. Then the kernel of $\mu$ is $\GL(\alpha)_\mu$ and the image is finite as it is contained in the set of $d^{th}$ roots of unity. This means that the identity component of $\GL(\alpha)_\sigma$ is $\GL(\alpha)_\mu$. Thus $\GL(\alpha)_\sigma$-semistability/polystability/stability is the same as $\GL(\alpha)_\mu$-semistability/polystability/stability by Lemma~\ref{lem:conn-component-stability}. By Proposition~\ref{gen.stable.transfer}, we deduce that generic semistability/polystability/stability for the action of $\GL(\alpha)_\mu$ on $\Rep(\Delta[k], \alpha))$ is equivalent to generic semistability/polystability/stability for the action of $(\GL(\alpha)_\mu)_\R$ on $(\Rep(\Delta[k], \alpha)))_\R$.

A similar argument as above will show that $(\GL(\alpha)_\mu)_\R$ is a subgroup of finite index in $(\GL(\alpha)_\sigma)_\R$. Hence their identity components are the same, so generic semistability/polystability/stability for the action of $(\GL(\alpha)_\mu)_\R$ on $(\Rep(\Delta[k], \alpha)))_\R$ is equivalent to generic semistability/polystability/stability for the action of $(\GL(\alpha)_\sigma)_\R$ on $(\Rep(\Delta[k], \alpha)))_\R$, again by Lemma~\ref{lem:conn-component-stability}.

%(see also \cite[Remark~3.11]{AKRS} for a similar argument). We will give a brief sketch.
 
Thus, $\alpha$ being $\sigma$-semistable/polystable/stable is the same as generic\\
semistability/polystability/stability for the action of $(\GL(\alpha)_\sigma)_\R$ on $(\Rep(\Delta[k], \alpha)))_\R$. Thus Proposition~\ref{prop:ipca-to-gaussianmodel} and Corollary~\ref{cor:ipca-Malpha} immediately give the first two statements and the backwards direction of the last statement. So, only the forward direction of the last statement remains, which we prove by contradiction.

Suppose $\alpha$ is not $\sigma$-stable. If $\alpha$ is not $\sigma$-polystable, then we do not have generic existence of MLE does not exist so clearly we cannot have generic existence of a unique MLE. So, we can assume $\alpha$ is not $\sigma$-stable, but $\sigma$-polystable. So we do not have generic indecomposability in $\Rep(Q,\alpha)_\C$. By Lemma~\ref{lem:transfer-gen-indec}, we do not have generic indecomposability in $\Rep(Q,\alpha)_\R$. By Lemma~\ref{lem:dec-compact}, we do not have generically have a compact stabilizer. Thus, by Lemma~\ref{lem:rect-AKRS}, we do have generically have the existence of a unique MLE.
\end{proof}

\begin{remark}
In the last paragraph of the proof above, we are careful to use phrases like ``we do not have generic property P" instead of ``we generically have the property of not P." There is a subtle difference in the two. The former says that the set of points satisfying P does not contain a Zariski-open dense subset and the latter says that the set of points satisfying P is contained in a lower dimensional subvariety. 
\end{remark}

\subsection*{Acknowledgements}

CF thanks Ankur Moitra for mentioning iPCA and Tiffany M. Tang for helpful conversations. The authors thank Michael Walter for correcting a subtle point about the time complexity of our algorithm.

\bibliographystyle{plain}
\bibliography{Notes}

\begin{thebibliography}{10}

\bibitem{AKRS}
Carlos Am{\'e}ndola, Kathl{\'e}n Kohn, Philipp Reichenbach, and Anna Seigal.
\newblock Invariant theory and scaling algorithms for maximum likelihood
  estimation.
\newblock {\em arXiv preprint arXiv:2003.13662}, 2020.

\bibitem{BFGOWW}
Peter B\"{u}rgisser, Cole Franks, Ankit Garg, Rafael Oliveira, Michael Walter,
  and Avi Wigderson.
\newblock Efficient algorithms for tensor scaling, quantum marginals, and
  moment polytopes.
\newblock In {\em 59th {A}nnual {IEEE} {S}ymposium on {F}oundations of
  {C}omputer {S}cience---{FOCS} 2018}, pages 883--897. IEEE Computer Soc., Los
  Alamitos, CA, 2018.

\bibitem{bi-LR}
Peter B{\"u}rgisser and Christian Ikenmeyer.
\newblock Deciding positivity of littlewood--richardson coefficients.
\newblock {\em SIAM Journal on Discrete Mathematics}, 27(4):1639--1681, 2013.

\bibitem{chindris2022membership}
Calin Chindris, Brett Collins, and Daniel Kline.
\newblock Membership in moment cones, quiver semi-invariants, and generic
  semi-stability for bipartite quivers.
\newblock {\em arXiv preprint arXiv:2211.01990}, 2022.

\bibitem{DK}
Harm Derksen and Gregor Kemper.
\newblock {\em Computational invariant theory}, volume 130 of {\em
  Encyclopaedia of Mathematical Sciences}.
\newblock Springer, Heidelberg, enlarged edition, 2015.
\newblock With two appendices by Vladimir L. Popov, and an addendum by Norbert
  A'Campo and Popov, Invariant Theory and Algebraic Transformation Groups,
  VIII.

\bibitem{DM-arbchar}
Harm Derksen and Visu Makam.
\newblock Generating invariant rings of quivers in arbitrary characteristic.
\newblock {\em J. Algebra}, 489:435--445, 2017.

\bibitem{DM}
Harm Derksen and Visu Makam.
\newblock Polynomial degree bounds for matrix semi-invariants.
\newblock {\em Adv. Math.}, 310:44--63, 2017.

\bibitem{DM-si}
Harm Derksen and Visu Makam.
\newblock Degree bounds for semi-invariant rings of quivers.
\newblock {\em J. Pure Appl. Algebra}, 222(10):3282--3292, 2018.

\bibitem{DM-exp}
Harm Derksen and Visu Makam.
\newblock An exponential lower bound for the degrees of invariants of cubic
  forms and tensor actions.
\newblock {\em Adv. Math.}, 368:107136, 25, 2020.

\bibitem{derksen2020maximum}
Harm Derksen and Visu Makam.
\newblock Maximum likelihood estimation for matrix normal models via quiver
  representations.
\newblock {\em arXiv preprint arXiv:2007.10206}, 2020.

\bibitem{derksen2020maximum2}
Harm Derksen, Visu Makam, and Michael Walter.
\newblock Maximum likelihood estimation for tensor normal models via castling
  transforms.
\newblock {\em arXiv preprint arXiv:2011.03849}, 2020.

\bibitem{DSW}
Harm Derksen, Aidan Schofield, and Jerzy Weyman.
\newblock On the number of subrepresentations of a general quiver
  representation.
\newblock {\em Journal of the London Mathematical Society}, 76(1):135--147,
  2007.

\bibitem{DW-LR}
Harm Derksen and Jerzy Weyman.
\newblock Semi-invariants of quivers and saturation for littlewood-richardson
  coefficients.
\newblock {\em Journal of the American Mathematical Society}, 13(3):467--479,
  2000.

\bibitem{DW-comb}
Harm Derksen and Jerzy Weyman.
\newblock The combinatorics of quiver representations.
\newblock In {\em Annales de l'Institut Fourier}, volume~61, pages 1061--1131,
  2011.

\bibitem{DW-book}
Harm Derksen and Jerzy Weyman.
\newblock {\em An introduction to quiver representations}, volume 184.
\newblock American Mathematical Soc., 2017.

\bibitem{derksen2010quivers}
Harm Derksen, Jerzy Weyman, and Andrei Zelevinsky.
\newblock Quivers with potentials and their representations ii: applications to
  cluster algebras.
\newblock {\em Journal of the American Mathematical Society}, 23(3):749--790,
  2010.

\bibitem{DZ}
M{\'a}ty{\'a}s Domokos and Alexander~N Zubkov.
\newblock Semi-invariants of quivers as determinants.
\newblock {\em Transformation groups}, 6(1):9--24, 2001.

\bibitem{forbes2013explicit}
Michael~A Forbes and Amir Shpilka.
\newblock Explicit noether normalization for simultaneous conjugation via
  polynomial identity testing.
\newblock In {\em Approximation, Randomization, and Combinatorial Optimization.
  Algorithms and Techniques}, pages 527--542. Springer, 2013.

\bibitem{franks2021near}
Cole Franks, Rafael Oliveira, Akshay Ramachandran, and Michael Walter.
\newblock Near optimal sample complexity for matrix and tensor normal models
  via geodesic convexity.
\newblock {\em arXiv preprint arXiv:2110.07583}, 2021.

\bibitem{franks2020rigorous}
William~Cole Franks and Ankur Moitra.
\newblock Rigorous guarantees for tyler's m-estimator via quantum expansion.
\newblock In {\em Conference on Learning Theory}, pages 1601--1632. PMLR, 2020.

\bibitem{GGOW16}
Ankit Garg, Leonid Gurvits, Rafael Oliveira, and Avi Wigderson.
\newblock A deterministic polynomial time algorithm for non-commutative
  rational identity testing.
\newblock In {\em 57th {A}nnual {IEEE} {S}ymposium on {F}oundations of
  {C}omputer {S}cience---{FOCS} 2016}, pages 109--117. IEEE Computer Soc., Los
  Alamitos, CA, 2016.

\bibitem{GGOW-BL}
Ankit Garg, Leonid Gurvits, Rafael Oliveira, and Avi Wigderson.
\newblock Algorithmic and optimization aspects of brascamp-lieb inequalities,
  via operator scaling.
\newblock {\em Geometric and Functional Analysis}, 28(1):100--145, 2018.

\bibitem{Ginzburg}
Victor Ginzburg.
\newblock Lectures on nakajima's quiver varieties.
\newblock {\em arXiv preprint arXiv:0905.0686}, 2009.

\bibitem{GLS-book}
Martin Gr\"{o}tschel, L\'{a}szl\'{o} Lov\'{a}sz, and Alexander Schrijver.
\newblock {\em Geometric algorithms and combinatorial optimization}, volume~2
  of {\em Algorithms and Combinatorics}.
\newblock Springer-Verlag, Berlin, second edition, 1993.

\bibitem{HW14}
Pavel Hrubes and Avi Wigderson.
\newblock Non-commutative arithmetic circuits with division.
\newblock In Moni Naor, editor, {\em Innovations in Theoretical Computer
  Science, ITCS'14, Princeton, NJ, USA, January 12-14, 2014}, pages 49--66.
  {ACM}, 2014.

\bibitem{IQS2}
G\'{a}bor Ivanyos, Youming Qiao, and K.~V. Subrahmanyam.
\newblock Constructive non-commutative rank computation is in deterministic
  polynomial time.
\newblock {\em Comput. Complexity}, 27(4):561--593, 2018.

\bibitem{Kac}
Victor~G Kac.
\newblock Infinite root systems, representations of graphs and invariant
  theory.
\newblock {\em Inventiones mathematicae}, 56(1):57--92, 1980.

\bibitem{Kac2}
Victor~G Kac.
\newblock Infinite root systems, representations of graphs and invariant
  theory, ii.
\newblock {\em Journal of algebra}, 78(1):141--162, 1982.

\bibitem{Keller}
Bernhard Keller.
\newblock Cluster algebras, quiver representations and triangulated categories.
\newblock {\em arXiv preprint arXiv:0807.1960}, 2008.

\bibitem{King}
Alastair~D King.
\newblock Moduli of representations of finite dimensional algebras.
\newblock {\em The Quarterly Journal of Mathematics}, 45(4):515--530, 1994.

\bibitem{knutson1999honeycomb}
Allen Knutson and Terence Tao.
\newblock The honeycomb model of $GL_n(\mathbb{C})$ tensor products
  i: Proof of the saturation conjecture.
\newblock {\em Journal of the American Mathematical Society}, 12(4):1055--1090,
  1999.

\bibitem{GCTV}
Ketan~D. Mulmuley.
\newblock Geometric complexity theory {V}: {E}fficient algorithms for {N}oether
  normalization.
\newblock {\em J. Amer. Math. Soc.}, 30(1):225--309, 2017.

\bibitem{mulmuley2012geometric}
Ketan~D Mulmuley, Hariharan Narayanan, and Milind Sohoni.
\newblock Geometric complexity theory iii: on deciding nonvanishing of a
  littlewood--richardson coefficient.
\newblock {\em Journal of Algebraic Combinatorics}, 36(1):103--110, 2012.

\bibitem{mulmuley2012arxiv}
Ketan~D. Mulmuley and Milind Sohoni.
\newblock Geometric complexity iii: on deciding positivity of
  littlewood-richardson coefficients.

\bibitem{Reineke}
Markus Reineke.
\newblock Moduli of representations of quivers.
\newblock {\em arXiv preprint arXiv:0802.2147}, 2008.

\bibitem{Ressayre}
Nicolas Ressayre.
\newblock Multiplicative formulas in schubert calculus and quiver
  representation.
\newblock {\em Indagationes Mathematicae}, 22(1-2):87--102, 2011.

\bibitem{SVd}
Aidan Schofield and Michel Van~den Bergh.
\newblock Semi-invariants of quivers for arbitrary dimension vectors.
\newblock {\em Indagationes Mathematicae}, 12(1):125--138, 2001.

\bibitem{tang2018integrated}
Tiffany~M Tang and Genevera~I Allen.
\newblock Integrated principal components analysis.
\newblock {\em arXiv preprint arXiv:1810.00832}, 2018.

\bibitem{tardos1986strongly}
{\'E}va Tardos.
\newblock A strongly polynomial algorithm to solve combinatorial linear
  programs.
\newblock {\em Operations Research}, 34(2):250--256, 1986.

\bibitem{tyler1987distribution}
David~E Tyler.
\newblock A distribution-free m-estimator of multivariate scatter.
\newblock {\em The annals of Statistics}, pages 234--251, 1987.

\bibitem{venturelli2019prehomogeneous}
Federico Venturelli.
\newblock Prehomogeneous tensor spaces.
\newblock {\em Linear and Multilinear Algebra}, 67(3):510--526, 2019.

\bibitem{werner2008estimation}
Karl Werner, Magnus Jansson, and Petre Stoica.
\newblock On estimation of covariance matrices with {Kronecker} product
  structure.
\newblock {\em IEEE Transactions on Signal Processing}, 56(2):478--491, 2008.

\bibitem{Whitney}
Hassler Whitney.
\newblock Elementary structure of real algebraic varieties.
\newblock {\em Annals of Mathematics}, 66(3):545--556, 1957.

\end{thebibliography}

\end{document}